\documentclass[english,10pt]{article}
\usepackage[papersize={19cm,25cm},body={15cm,21cm},centering]{geometry}
\usepackage[T1]{fontenc}
\usepackage[latin9]{inputenc}
\usepackage{amsmath}
\usepackage{amssymb}
\usepackage{url}
\usepackage{babel}
\usepackage{hyperref}
\usepackage{comment}
 \usepackage{graphicx}
\usepackage{epsfig,subfigure,color}
\usepackage{graphics}
 \usepackage{graphicx}
\usepackage{amsmath,amsfonts,amsthm,latexsym}
\usepackage{mathrsfs,amsbsy}
\usepackage{multirow}
\usepackage{epstopdf}
\usepackage{algorithm}
\usepackage{algorithmic}
\usepackage{cite}

\usepackage{nicefrac} 

\newtheorem{corollary}{Corollary}
\newtheorem{theorem}{Theorem}
\newtheorem{lemma}{Lemma}

\newtheorem{example}{Example}
\newtheorem{remark}{Remark}
\newtheorem{assumption}{Assumption}

\theoremstyle{definition}

\newcommand{\revise}[1]{{{\color{black} #1}}}

\usepackage{enumitem}

 \usepackage{color}

\definecolor{brightpink}{rgb}{1.0, 0.0, 0.5}

\definecolor{amethyst}{rgb}{1.0, 0.49, 0.0}
\definecolor{coquelicot}{rgb}{1.0, 0.22, 0.0}

\DeclareMathOperator{\rank}{rank}

\DeclareMathOperator{\diag}{diag}

\begin{document}
\title{A Provably-Correct and Robust Convex Model for \\ Smooth Separable NMF\thanks{The first two authors contributed equally to this work.}} 
\date{}

\author{
Junjun Pan\thanks{Department of Mathematics, Hong Kong Baptist University, Hong Kong. Email: junjpan@hkbu.edu.hk. J.Pan is supported in part by RGC GRF 12301404, startup grant 162994 and by Guang Dong Basic and Applied Basic Research Foundation (grant 2023A1515010973).}\qquad
Valentin Leplat\thanks{Institute of Data Science and Artificial Intelligence, Faculty of Computer Science and Engineering, Innopolis University. Email: v.leplat@innopolis.ru} \qquad 
Michael Ng\thanks{Department of Mathematics, Hong Kong Baptist University, Hong Kong. Email: michael-ng@hkbu.edu.hk. M.Ng's research is supported by the
GDSTC: Guangdong and Hong Kong Universities "1+1+1" Joint Research Collaboration Scheme project No.: 2025A0505000007, National Key Research and Development Program of China under Grant 2024YFE0202900, RGC GRF 12300125.} \qquad 
Nicolas Gillis\thanks{Department of Mathematics and Operational Research, Universit\'e de Mons, Rue de Houdain 9, 7000 Mons, Belgium.
Email: nicolas.gillis@umons.ac.be.
This work was supported by the European Research Council (ERC consolidator, eLinoR, no 101085607).} 
}

\maketitle
\begin{abstract}

Nonnegative matrix factorization (NMF) is a linear dimensionality reduction technique for nonnegative data, with applications such as hyperspectral unmixing and topic modeling. NMF is a difficult problem in general (NP-hard), and its solutions are typically not unique. 
To address these two issues, additional constraints or assumptions are often used. 
In particular, separability assumes that the basis vectors in the NMF are equal to some columns of the input matrix. In that case, the problem is referred to as separable NMF (SNMF) and can be solved in polynomial-time with robustness guarantees, while identifying a unique solution. 
However, in real-world scenarios, due to noise or variability, multiple data points
may lie near the basis vectors, which SNMF does not leverage. 
In this work, we rely on the smooth separability assumption, which assumes that each basis vector is close to multiple data points. We explore the properties of the corresponding problem, referred to as smooth SNMF (SSNMF), and examine how it relates to SNMF and orthogonal NMF. We then propose a convex model for SSNMF and show that it provably recovers the sought-after factors, even in the presence of noise. 
We finally adapt an existing fast gradient method to solve this convex model for SSNMF, and show that it compares favorably with state-of-the-art methods on both synthetic and hyperspectral datasets.


\end{abstract}


\section{Introduction}
Given a nonnegative matrix $M \in \mathbb{R}^{m\times n}_+$ and an integer factorization rank $r$,
nonnegative matrix factorization (NMF) is the problem of computing
$W \in \mathbb{R}^{m\times r}_+$ and
$H \in \mathbb{R}^{r\times n}_+$
such that $M \approx WH$.
Typically, the columns of the input matrix $M$ correspond to data points (such as images of pixel intensities or documents of word counts) and NMF allows one to perform linear dimensionality reduction.
In fact, we have $M(:,j) \approx \sum_{k=1}^r W(:,k) H(k,j)$ for all $j$, where $M(:,j)$ denotes the $j$th column of $M$.
This means that the data points are approximated by points within an $r$-dimensional subspace spanned by the columns of $W$. The nonnegativity constraints lead to easily interpretable factors with applications such as image processing, text mining, hyperspectral unmixing and audio source separation; see, e.g., 
\cite{Cich09, xiao2019uniq, gillis2020nonnegative} and the references therein.

NMF is known to be NP-hard in general~\cite{vavasis2009complexity}, and its solutions are often non-unique; see~\cite{xiao2019uniq, gillis2020nonnegative}. 
These challenges have motivated the introduction of additional assumptions, such as orthogonality or separability, to enable efficient computation and ensure unique solutions. 
Though these assumptions have led to significant theoretical and practical advances, they often rely on strong conditions.
For example, 
orthogonal NMF (ONMF) enforces $H$ to have orthonormal rows, 
implying that each column of the input data matrix $M$ is approximated by a scalar multiple of some column of the basis matrix $W$~\cite{Ding05}. 
It cannot deal with data points that lie within the cone spanned by the columns of $W$. 
Another example is separable NMF (SNMF) which requires that there exist at least one column of $M$ near each column of the basis matrix $W$~\cite{Dono04}. This condition simplifies the problem, making it solvable in polynomial time with a unique solution~\cite{arora2012computing}. 
However, separability does not leverage the fact that there may be multiple data points near each column of $W$ due to noise or variability, limiting the applicability of SNMF in noisy or complex real-world scenarios.

Motivated by this observation, smooth SNMF (SSNMF)~\cite{bhattacharyya2020finding, 
bhattacharyya2021finding, 
bakshilearning, NADISIC2023174} has emerged recently as a promising alternative, by taking advantage of the fact that there might be multiple data points near each column of $W$. This assumption, also known as the proximate latent points assumption~\cite{bhattacharyya2020finding, 
bhattacharyya2021finding, 
bakshilearning}, has been shown to improve robustness to noise and spectral variability, particularly in applications like hyperspectral unmixing and topic modeling~\cite{NADISIC2023174}. 


\paragraph{Contribution and outline}  

In this paper, we propose the first  convex model for SSNMF, along with an effective first-order algorithm. Our contributions are as follows.  
\begin{enumerate}
\item We first examine how smooth separable NMF (SSNMF) is related to separable NMF (SNMF) and orthogonal NMF (ONMF), providing a unified perspective on these models.

\item We introduce a convex formulation of smooth separable NMF, and propose to use efficient and scalable first-order optimization methods to solve it.  

\item We show that our algorithmic pipeline provably recovers the ground-truth solution, even in the presence of noise. 

\item We show the effectiveness of our approach through extensive experiments on synthetic and real-world datasets.  

\end{enumerate}

The remainder of the paper is organized as follows.   
Section~\ref{sec:2_SNMF_ONMF_SSNMF} provides background on SNMF, discusses its limitations, and introduces SSNMF. 
It then establishes a key connection between ONMF and SNMF, leading to a new definition for ONMF. 
Section~\ref{sec:Model} uses these insights to propose a  convex model for SSNMF (CSSNMF). It then studies the recovery guarantees offered by CSSNMF, even in the presence of noise. 
Section~\ref{sec:Algo} presents our two-step algorithmic pipeline to tackle CSSNMF. 
Section~\ref{sec:numexp} illustrates the effectiveness of CSSNMF through experimental results on synthetic and hyperspectral images.







\paragraph{Notation} 
Given a real $m$-by-$n$ matrix $M \in \mathbb{R}^{m\times n}$, $\|M\|_F^2 = \sum_{i,j} M(i,j)^2$ is the squared Frobenius norm where $M(i,j)$ is the entry of $M$ at position $(i,j)$. The $\ell_1$ norm of a matrix is $\|M\|_1 = \max_{\|x\|_1 \leq 1} \|Mx\|_1 = \max_j \|M(:,j)\|_1$ where, for a vector $x$,  $\|x\|_1 = \sum_i |x_i|$ is the $\ell_1$ norm of $x$. (Note that $\|M\|_1$ is not the component-wise $\ell_1$ norm of $M$.) 
Given a vector $x$, its $\ell_{\infty}$ norm is $\|x\|_\infty = \max_i |x_i|$, and its squared $\ell_2$ norm is $\|x\|_2^2 = \sum_i x_i^2$. 
Given a matrix $M$, $\diag(M)$ is the vector containing the diagonal entries of $M$. 
The vector of all ones is denoted by $e$, while $e_k$ denotes the $k$th unit vector; their dimensions will be clear from the context. The identity matrix of dimension $r$ is denoted as $I_r$, and $0_{m \times n}$ denotes the $m$-by-$n$ all-zero matrix. 
A permutation matrix $\Pi \in \{0,1\}^{n \times n}$ is obtained by a permutation of the columns of $I_n$. Given a set $\mathcal K$, $|\mathcal K|$ denotes its cardinality. 
Given an index set $\mathcal K$, $M(:,\mathcal K)$ is the submatrix containing the subset of columns indexed by $\mathcal K$, and similarly for the rows.

\section{Separable, Smooth and Orthogonal NMF}\label{sec:2_SNMF_ONMF_SSNMF}

In this section, we first recall the notion of SNMF and its practical limitations. We then introduce the smooth separability assumption, which extends the standard separability model by leveraging the existence of multiple data points near each basis vector. 
Next, we discuss the connections between SNMF and ONMF. We highlight key insights that will guide the development of our convex formulation in subsequent sections. 




\begin{assumption} \label{assum:stocNMF}
   The matrix $M \in \mathbb{R}^{m\times n}$ is such that 
   $M = WH + N$ where 
   \begin{itemize}
       \item $H \in \mathbb{R}^{r \times n}_+$ is column stochastic, that is, $H^\top e= e$;
       \item $N \in \mathbb{R}^{m\times n}$ is noise  with a column-wise $\ell_1$ bound:
       $\|N\|_1 = \max_j \|N(:,j)\|_1 \le \epsilon$, \revise{ for some $\epsilon \geq 0$.} 
   \end{itemize}
\end{assumption} 
We also need the following quantity that measures the conditioning of $W$: 
$$
\kappa(W) =\min_{k \in \{1,2,\dots,r\}} 
\min_{x\in \mathbb{R}^{r-1}_+}\|W(:,k) - W(:,\bar k) x\|_1,    
$$ 
where $\bar k = \{ 1,2,\dots,r \} \backslash \{k \}$.
Note that $\min_{k \neq j} \|W(:,k)- W(:,j)\|_1 \geq \kappa(W)$.  
\revise{In the remainder of this paper, we will denote $\kappa := \kappa(W)$, for simplicity. We will need that $\kappa > 0$; in fact, it is not possible to recover columns of $W$ that are in the convex cone generated by the others; they cannot be distinguished from data points~\cite{arora2012computing}. In fact, any NMF with $\kappa(W)=0$ can be reduced to an NMF with smaller rank, by removing the columns of $W$ contained in the convex cones of the others.}

The stochasticity of the columns of $H$ implies that each column of $M$ is a convex linear combination of the columns of $W$; see Figure~\ref{fig:ass1} for an illustration for $r=3$. 
\begin{figure}[ht!] 
\begin{center} 
\includegraphics[width=7cm]{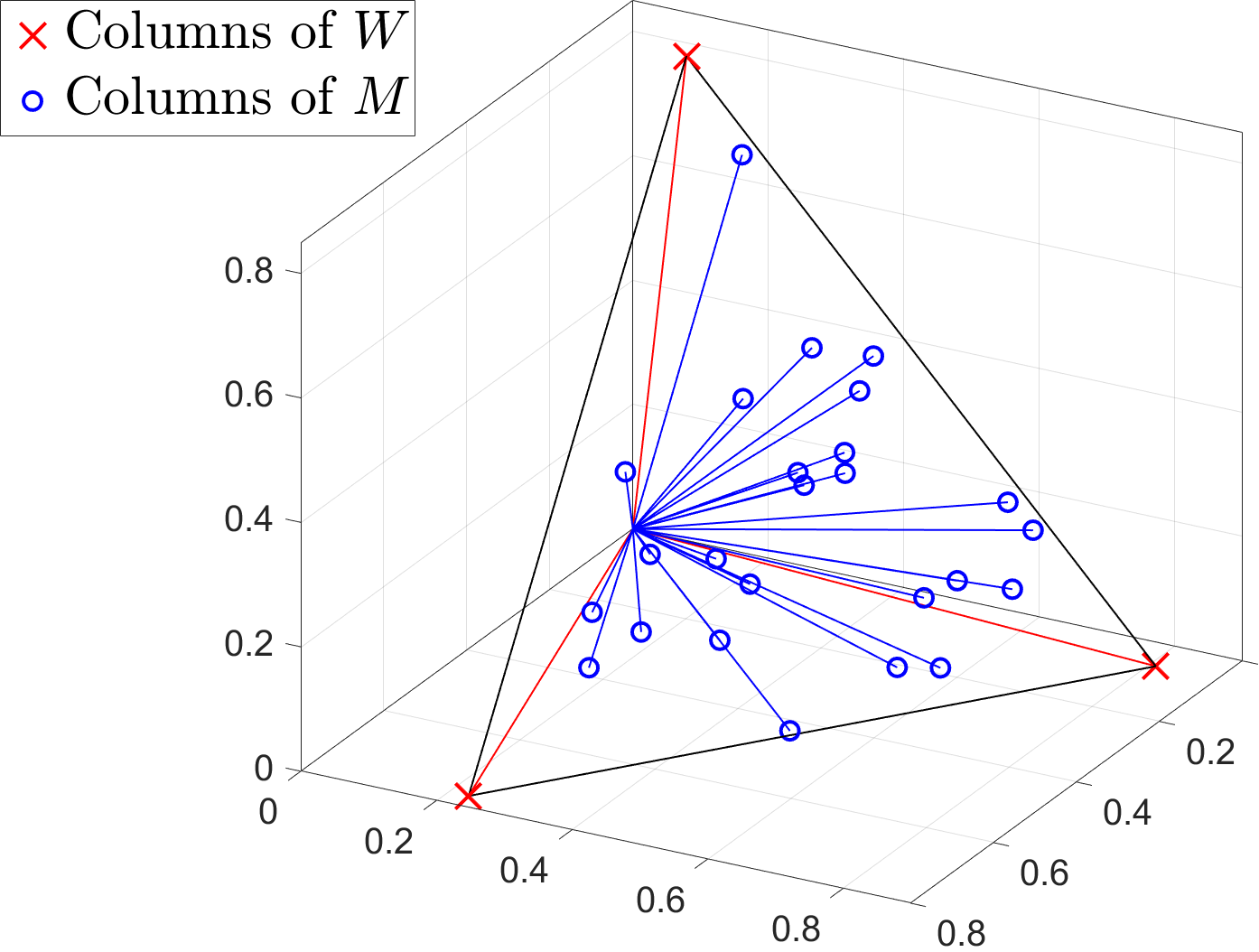}
\caption{Illustration of Assumption~\ref{assum:stocNMF} in the noiseless case ($N=0$) for $r=3$.}\label{fig:ass1}
 \end{center}
\end{figure} 
This is the so-called simplex-structured matrix factorization model~\cite{lin2018maximum, wu2021probabilistic, abdolali2021simplex}. 
In the noiseless case, this assumption can be made w.l.o.g.\ by renormalizing each column of $M$ to sum to one; see, e.g., the discussion in~\cite[Chapter 7]{gillis2020nonnegative}. 
This assumption also simplifies the presentation, but, as we will explain later, our convex model can be easily adapted to handle the case where $H$ is not column stochastic.

Note that $W$ is not required to be nonnegative, only $H$ is. The key is the nonnegativity in the linear combinations, not in the basis vectors. 

\subsection{Separable NMF (SNMF)}

A matrix $M\in \mathbb{R}^{m\times n}$ satisfying Assumption~\ref{assum:stocNMF} is separable if there exist an index set $\mathcal{K}$ of cardinality $r$ such that $W = M(:,\mathcal{K}) + N(:,\mathcal K)$.
Geometrically, separability  means that there exists a subset of $r$ columns of $M$ close  to the columns of $W$. 
Under the separability assumption, NMF can be solved in polynomial time, with robustness guarantees in the presence of noise. This line of work was pioneered by Arora et al.~\cite{arora2012computing, arora2016computing}, and many works followed, e.g.,  \cite{recht2012factoring, arora2013practical, gillis2013robustness, gillis2014fast, mizutani2025endmember}; see~\cite[Chapter~7]{gillis2020nonnegative} for an overview.

The separability assumption is reasonable in several practical applications:
\begin{enumerate}
\item Hyperspectral unmixing. Each column of the data matrix contains the spectral signature of a pixel. Separability requires that for each material in the image, there exists a pixel containing only that material. Separability is known as  the pure-pixel assumption in the remote sensing literature and has generated a large body of literature; see, e.g., \cite{bioucas2012hyperspectral, ma2014signal} and the references therein. 

\item Audio source separation. The input matrix is a time-frequency amplitude spectrogram~\cite{fevotte2009nonnegative}. Separability requires that, for each source, there exists a moment in time when only that source is active~\cite{leplat2019minimum}. 

\item Document classification. The input matrix is a document-by-word count matrix. Separability requires that, for each topic, there exists at least one word discussing only that topic, an "anchor" word~\cite{arora2013practical}.  

\item Facial feature extraction. The input matrix is a face-by-pixel intensity matrix. Separability requires that, for each facial feature (such as eyes, noses and lips), there exists at least one pixel present only in that feature~\cite[page~209]{gillis2020nonnegative}.  

\end{enumerate}

    While the separability assumption simplifies the NMF problem and enables efficient algorithms, 
    it often 
    \revise{does not fully take advantage of the properties of 
     real-world data}. In fact, in many practical applications, such as hyperspectral unmixing or document classification, data points rarely correspond exactly to pure sources. Instead, multiple data points may lie near each true basis element due to noise, spectral variability, or other distortions. This motivates the need for a more flexible framework that allows for groups of data points to approximate each basis element, leading to the smooth separability assumption, 
    which we discuss in the next section. 
    \revise{Taking advantage of multiple data points near the basis vectors will allow us to better estimate these basis vectors, and hence also provide better low-rank approximations of the input matrix.}

\paragraph{Separability using self-dictionary and convex relaxation} 
Let us now present an equivalent definition of separability that will be particularly useful in this paper. 
It relies on the so-called self-dictionary approach, which learns a set of basis vectors from the data itself. 
It was originally proposed in~\cite{esser2012convex, recht2012factoring, elhamifar2012see} in order to design convex formulations for matrix factorization problems. 

The matrix $M$ \revise{satisfying Assumption~\ref{assum:stocNMF}} is separable if there exist some permutation matrix $\Pi\in \{0,1\}^{n\times n}$ and a nonnegative matrix $H'\in \mathbb{R}^{r\times (n-r)}_+$ such that
\begin{equation*}
M\Pi=M\Pi\left(
           \begin{array}{cc}
             I_r  & H'  \\
              0_{n-r,r} & 0_{n-r,n-r}  \\
           \end{array}
         \right),
\end{equation*}
where $0_{r,p}$ is the matrix of all zeros of dimension $r$ by $p$.
In fact, under an appropriate permutation, the first $r$ columns of $M$ correspond to the columns of $W$ while the last $n-r$ columns are convex combinations of these first $r$ columns. Equivalently, we have
\begin{equation} \label{eq:convsep}
M \; = \; M \; \underbrace{ \Pi \left(
           \begin{array}{cc}
             I_r  & H'  \\
              0_{n-r,r} & 0_{n-r,n-r}  \\
           \end{array}
         \right) \Pi^{T} }_{X \in \mathbb{R}^{n \times n}}.
\end{equation} 
This means that $M \in \mathbb{R}^{m \times n}$ is separable if there exists a column-stochastic matrix  $X \in \mathbb{R}^{n \times n}$ 
with $r$ non-zero rows such that $M = MX$. 
\revise{In fact, under Assumption~\ref{assum:stocNMF} with $\epsilon=0$ and $\kappa(W)>0$, any such matrix $X$ must, up to permutation, contain the identity as a submatrix, since each column of $W$ can only be reconstructed by itself.} 

This observation motivates the following convex model for SNMF~\cite{gillis2014robust}, which is an improvement over~\cite{recht2012factoring}: Given $M \in \mathbb{R}^{m \times n}$ satisfying Assumption~\ref{assum:stocNMF} and separable, solve 
\begin{equation} \label{modelsepNMF} 
  \min_{X \in \mathbb{R}^{n \times n}} 
  \; e^\top \diag(X) 
  \quad \mbox{s.t.} \quad \|M-MX\|_1 \leq 2 \epsilon, \quad 
  \revise{0 \leq} \ X(i,j) \leq X(i,i) \leq 1, \quad \forall i, j. 
\end{equation} 
The objective of~\eqref{modelsepNMF} minimizes the trace of $X$. Together with the constraint $X(i,j) \leq X(i,i)$, this means that \eqref{modelsepNMF} minimizes $\sum_{i=1}^n \|X(i,:)\|_\infty$ which is a convex relaxation of the number of non-zero rows of $X$. 

The optimal solution of~\eqref{modelsepNMF} allows one to identify the $r$ columns of $M$ corresponding to the columns of $W$, even in the presence of noise. 
 For simplicity, we assume that  there are no duplicated columns of $W$ in the data set. 
The case with duplicated vertices also can be handled but requires a more involved analysis and proper post-processing of the optimal solution of~\eqref{modelsepNMF}; see~\cite[Theorem~7]{gillis2014robust}.  
Recall that for $M$  separable, $M = W[I_r, H'] \Pi$. 
Mathematically, the absence of duplicated columns of $W$  requires $\beta = \max_{k,j} H'(k,j) < 1$. 

We can now state the robustness result of solving~\eqref{modelsepNMF} for SNMF. 
\begin{theorem}\cite[Theorem~2]{gillis2014robust} \label{th:robustsepNMF}
    Let $M = WH+N$ satisfy Assumption~\ref{assum:stocNMF} and $M$ be separable, that is, 
    $M = W[I_r,  H']\Pi + N$ for some permutation matrix $\Pi$, \revise{and with $\kappa := \kappa(W) > 0$}. 
    Let also 
    $\max_{k,j} H'(k,j) = \beta < 1$, $X^*$ be an optimal solution of~\eqref{modelsepNMF}, and $\mathcal{K} = \{ k \ | \ X^*(k,k) \geq 1/2 \}$. 
    If 
\[
\epsilon < \frac{\kappa (1-\beta)}{20}, 
\]
then $|\mathcal{K}| = r$ and 
    \[ 
 \|W - {M}(:,\mathcal{K}) \Pi' \|_1 \leq \epsilon \quad \text{ for some permutation $\Pi'$.} 
    \]  
\end{theorem}

Note that solving~\eqref{modelsepNMF} does not require to know  a priori the value of $r$, and Theorem~\ref{th:robustsepNMF} allows one to recover it automatically. 
The noise level, $\epsilon$, is needed as an input. However, it does not need to be known very precisely, the robustness analysis can be adapted to the constraints $\| M - MX \|_1 \leq \rho  \epsilon$ for any $\rho > 0$; see \cite[Theorem~2]{gillis2014robust} for more details. For simplicity of the analysis, we stick to the constraint $\| M - MX \|_1 \leq 2 \epsilon$ in this paper. 

In practice, because of high noise levels, one might select columns of $M$ by taking the $r$ largest diagonal entries of $X$, or use more sophisticated postprocessing strategies~\cite{recht2012factoring, gillis2013robustness, gillis2014robust, mizutani2025endmember}.

\subsection{Smooth Separable NMF (SSNMF)}\label{ssnmf} 

In many applications, multiple data points are close to each column of $W$, and data points are often spread around them, due to noise or variability. 
The smooth separability assumption will leverage this observation explicitly, and is defined as follows.  


\color{black}
\begin{assumption} \label{assum:SSNMF}
Let $M = WH+N$ satisfy Assumption~\ref{assum:stocNMF}.
For each $t=1,2,\dots,r$, there exists a nonempty index set
$\mathcal{S}_t \subset \{1,\dots,n\}$, with $p_t = |\mathcal{S}_t| \ge 1$, such that 
$H(:,j) = e_t$ for all $j \in \mathcal{S}_t$, that is, $M(:,j)=W(:,t) + N(:,j)$. 
\end{assumption}

\color{black}

 This assumption is realistic in noisy or complex scenarios, as it accounts for the natural variability in real-world data. 
 {\color{black} Unlike SNMF that looks for a single representative for each column of $W$,}
 SSNMF will leverage the fact that more than one data point can be located near each column of $W$ by aggregating several columns of $M$ to estimate each column of $W$. 
 This aggregation can be done using operations such as averaging or taking the median, which improves robustness to noise and outliers. 
 \revise{In a nutshell, SSNMF allows one to provide better estimates for $W$, and hence compute NMFs with lower reconstruction errors.}

In~\cite{bhattacharyya2020finding, 
bhattacharyya2021finding, 
bakshilearning, NADISIC2023174}, authors propose to solve this problem in two steps: 
\begin{enumerate}
    \item Compute $W$: identify  
 $\{M(:,j)\}_{j \in \mathcal{S}_t}$ \revise{for $t=1,2,\dots,r$} with dedicated routines and then obtain $W$ with an aggregation function such as the average or the median, 
 
 \item Compute $H$: solve the nonnegative least-squares (NNLS) problem $\min_{H \geq 0} \|M - WH \|_F^2$. 
\end{enumerate} 

    A key limitation of such algorithms is the requirement to specify the numbers $p_t$ of columns of $M$ that should be aggregated to obtain each column of $W$, that is, specify the size of each $\mathcal{S}_t$ {\color{black}for $t=1,2,\cdots, r$.}
    Hence, for simplicity, the same number is used for all $t$'s, that is, $p_t = p$ for all $t$ for some $p$.
    However, this parameter is critical to the performance of these methods, yet it is often difficult to determine in practice. 
    In fact, \cite{NADISIC2023174} mentions this issue as future research.  
    In contrast, as we will show later in this paper, our convex formulation for SSNMF can automatically identify the values of $p_t$, as well as the rank $r$, without requiring prior knowledge, similarly to the convex model for SNMF that does not need an explicit estimate of the rank $r$.  
    

\subsection{Orthogonal NMF (ONMF)}\label{sec:ONMF}

We first recall the definition of ONMF and then explore its connection with SSNMF. While these models serve different purposes, their relationship provides key insights for our new model.

In addition to the nonnegative constraint, ONMF also requires that the rows of $H$ are orthonormal, that is, $HH^\top=I_r$. Since $H \geq 0$, this means that each column of $H$ has at most one non-zero entry, hence ONMF is equivalent to a clustering problem~\cite{Ding05, Pomp14, Pan18}.

We note in Assumption~\ref{assum:stocNMF} that $H$ is column stochastic, so under this assumption, orthonormality of $H$ should be relaxed to  $H'(k,:)^\top H'(t,:) = 0$ for all $k \neq t$.
Interestingly, ONMF is a special case of SSNMF where all columns of $M$ correspond to a column of $W$. Hence SSNMF can be viewed as a model mid-way between SNMF and ONMF. 
Moreover, for a matrix satisfying smooth separability {\color{black}(that is, Assumption~\ref{assum:SSNMF}), its} submatrix 
$M\left( : , \cup_{t=1}^r \mathcal S_t \right)$ admits an ONMF. {\color{black} It is therefore interesting to understand the structure of ONMF from the perspective of the self-dictionary; which will lead to our convex model of SSNMF.}

\paragraph{ONMF and self-dictionary}

Recall that, in the self-dictionary model $M = MX$ under separability, $X$ has $r$ nonzero rows. This also applies to ONMF. In ONMF, each column of $M$ is multiple of some column of $W$, and hence ONMF is more constrained than SNMF. Similarly, inspired by self-dictionary models,  we have the following result.  

\begin{theorem} \label{lem:structONMF}
Let $M = WH$ satisfy Assumption~\ref{assum:stocNMF} with $\epsilon = 0$, that is,  
$N= 0$. Let also 
\mbox{$H(:,j) \neq 0$} for all $j$, 
\mbox{$H(t,:) \neq 0$} for all $t$, and 
$\kappa(W) > 0$.  Then $H$ has orthogonal rows, that is, \mbox{$H(k,:)^\top H(t,:) = 0$} for all $k \neq t$ if and only if $M$ is separable and there exists $X$ such that $M = MX$, where $X \geq 0$ and $X$ has the form 
\begin{equation}\label{eq:Xonmf} 
X =  \Pi B \Pi^T = \Pi\left(
    \begin{array}{cccc}
      B_1 &  0 & \cdots & 0 \\
      0 &  B_2  & \cdots & 0 \\
      \vdots&\vdots&\ddots&\vdots\\
      0&0&\cdots& B_r 
    \end{array}
  \right)\Pi^T,  
\end{equation} 
for some permutation $\Pi$, and each block matrix $B_t \in \mathbb{R}^{p_t \times p_t}$ can be \emph{any} column-stochastic matrix. 
\end{theorem}
\begin{proof}
First, let us simplify the notation by noting that  
    {\color{black}$M = MX = M \Pi B \Pi^\top$} if and only if $M \Pi = M \Pi B$. Hence we assume w.l.o.g.\ that the columns of $M$ are permuted so that $M = MB$ where $B$ is the block matrix as defined above. 

$\Rightarrow$ 
Since $H$ is column stochastic and has no zero column or row, $H(t,:)^\top H(k,:) = 0$ for all $k \neq t$ implies that $H(:,j) = e_t$ for some $t$, so that $M(:,j) = W(:,t)$. Since $H$ has no zero rows, this implies that $M$ is separable. 
Let us define $S_t = \{ j \ | \ M(:,j) = W(:,t) \}$. 
Let $X \geq 0$ be such that $M = MX$ hence $M(:,j) = M X(:,j)$ for all $j$. 
Since $\kappa(W) > 0$, no column of $W$ can be written as a \revise{nonnegative} linear combination of the other columns, by definition. 
Hence $X(i,j) = 0$ for all $j \in S_t$ and $i \notin S_t$, which implies that $X$ has the block structure as in~\eqref{eq:Xonmf} where $p_t = |\mathcal S_t|$. 
Since $M(:,j) = M(:,j')$ for all $j,j' \in \mathcal S_t$, the diagonal blocks of $X$, $X(\mathcal S_t, \mathcal S_t)$, can be any column stochastic matrix, and nothing else since $X \geq 0$ (for any non-zero vector $x = \sum_i \alpha_i x$ with $\alpha \geq 0$, we must have $\sum_{i} \alpha_i = 1$). 

$\Leftarrow$ We have $M = MB$ with $B$ as in~\eqref{eq:Xonmf}.  
Let the index set $\mathcal S_t$ be such that $B(\mathcal S_t,\mathcal S_t) = B_t$ with $|\mathcal S_t| = p_t$. 
Since $M = MB$ for any column stochastic $B_i$'s, we can choose $B_t(1,:) = e^\top$ and $B_t(j,:) = 0$ for all $j \neq 1$. 
This implies that $M(:,j) = M(:,i)$ for all $(i,j) \in \mathcal S_t$ for all $t=1,2,\dots,r$, that is, the columns of $M$ are made of $r$ clusters, each containing $p_t$ copies of the same column. 
Since $M$ is separable and $\kappa(W) > 0$ (implying that the columns of $W$ are distinct), 
these $r$ clusters must each correspond to one column of $W$, hence 
$M = WH$, with $H = [e_1 \dots e_1, e_2 \dots e_2, \dots, e_r \dots e_r]$ where there are $p_t$ copies of $e_t$, and $H$ is orthogonal. 
\end{proof}

\section{Convex Model for SSNMF and Recovery Guarantees}\label{sec:Model}

Building on the insights from the previous section, we now introduce a convex formulation for SSNMF. Our model generalizes SNMF by incorporating structured redundancy among basis elements while maintaining the benefits of convex optimization. This new framework eliminates the need for an explicit choice of the number of basis vectors, and the number of data points clustered around each basis vector. 
 
Then, we provide recovery guarantees in the presence of noise, similar to Theorem~\ref{th:robustsepNMF} for SNMF. 
In Section~\ref{sec:Algo}, we will propose a practical algorithm relying on this convex model.

\subsection{A Convex Model for SSNMF}\label{sec:CSSNMF_model}

Motivated by SNMF in \eqref{eq:convsep}  and  Theorem~\ref{lem:structONMF} about the structure of ONMF solutions,  we propose the following convex optimization problem:  
\begin{align} \label{eq:convexmodelSSNMF}
  \min_{X \in \mathbb{R}^{n \times n}} \| \diag(X) \|_2^2 
  \quad \text{ such that }     &  \quad \|M-MX\|_1 \leq 2 \epsilon,  \\ 
  & \quad 0 \leq X(i,j) \leq X(i,i) \leq 1, \quad \forall~ i,~ j.
\end{align}
where $\| \diag(X) \|_2^2  = \sum_{i} X(i,i)^2$. 
 The feasible set for $X$ is the same as that of the convex model for SNMF in~\eqref{modelsepNMF}. As opposed to SNMF, \eqref{eq:convexmodelSSNMF} minimizes the sum of the squared diagonal entries of $X$, not their sum. This incites the solution   to spread the weights of the solution, $X$, into as many columns of $M$ as possible that are close to columns of $W$. \revise{ 
 In fact, when solving a system of linear equations in variable $x$, 
  minimizing the $\ell_1$ norm of $x$ incites it to be sparse, while minimizing its $\ell_2$ norm incites it to be dense; see, e.g., \cite{baraniuk2007compressive} and the references therein. Intuitively, the $\ell_1$ ball will touch the feasible set with one of its sparse vertices, while the $\ell_2$ ball is isotropic and hence the solution is unlikely to be sparse.} 
  
\revise{The following lemma, which will be useful to prove our recovery guarantees for~\eqref{eq:convexmodelSSNMF}, shows that trying to minimize the $\ell_2$ norm of a vector under a linear inequality finds a solution that spreads the weights among the entries of that vector.} 
    \begin{lemma} \label{lem:SOSsum} 
Let $\alpha > 0$. The optimal solution of the optimization problem 
\[
\min_{x \in \mathbb{R}^p} \sum^p_{i=1} x_i^2 \quad \text{ such that } \quad \sum^p_{i=1} x_i \geq \alpha, 
\]
is given by $x_i^* = \frac{\alpha}{p}$ for all $i$, with objective $f^* = \frac{\alpha^2}{p}$. 
\end{lemma}
\begin{proof} This can be easily checked using the optimality conditions of this convex optimization problem:  $\frac{\partial}{\partial  x}\big(\sum\limits^p_{i=1} x_i^2 +\lambda (\alpha-\sum\limits^p_{i=1} x_i)\big)=0$, that is, $2x = \lambda e$ where $\lambda \geq 0$ is the Lagrange multiplier of the constraint $\sum_i x_i \geq \alpha$. This implies that all entries of $x$ are equal to one another in an optimal solution. 
\end{proof}

\revise{
\begin{remark}[Other objectives] 
We could use other objective functions in~\eqref{eq:convexmodelSSNMF}
to achieve our goal. For example, one could use 
$p$-norms of $\diag(X)$ with $1 < p \leq +\infty$, to which Lemma~\ref{lem:SOSsum} also applies, with the same optimal solution (all entries of $x$ are equal to one another). This means that, in the noiseless case, this choice would not influence the optimal solution of~\eqref{eq:convexmodelSSNMF}. 
However, analyzing the influence of the choice of the objective in the presence of noise, both from a
theoretical and practical point of view, is a topic for further research. 
From a practical point of view, the advantage of the $\ell_2$ norm is that it is more easily handled computationally, since it leads to a linearly constrained quadratic program, and its square is Lipschitz smooth, which is crucial in the design of our first-order algorithm in Section~\ref{sec:Algo}. 

Note that the limiting case, $p = +\infty$, would not be useful as the optimal solution would not be unique, in particular $X=I_n$ is an optimal solution which does not bring any valuable information on the structure of $M$.   
\end{remark}
}

Let us give some intuition on the structure of the optimal solution of~\eqref{eq:convexmodelSSNMF} under the smooth separable Assumption~\ref{assum:SSNMF} in the noiseless case, that is, $N = 0$ and $\epsilon = 0$. The proof in the general case will be given in the next section. 

Because we are minimizing the sum of the squared diagonal entries,  $M = MX$ and $M$ satisfies Assumption~\ref{assum:SSNMF}, we must assign a total weight of 1 to each cluster of columns $\mathcal S_t$ to be able to reconstruct the columns of $M$ corresponding to the columns of $W$. By Lemma~\ref{lem:SOSsum}, the objective is minimized if we give exactly the same weight to each column in $\mathcal S_t$, namely $\frac{1}{p_t}$. 
Defining $\mathcal S = \cup_{t=1}^r S_t$, we know that $M(:,\mathcal{S})$ has an ONMF decomposition, and  $X(\mathcal S,\mathcal S)$ will have the block structure of~\eqref{eq:Xonmf}, by Lemma~\ref{lem:structONMF}. Combining these two observations,  $B_t = \frac{ee^\top}{p_t}$ will minimize the objective. 
Finally, the optimal solution of~\eqref{eq:convexmodelSSNMF}, in the noiseless case, will have the following form 
\[
X^*(i,:) = 
\left\{ 
\begin{array}{cl}
   0  & \text{ for } i \notin \mathcal S,  \\
   \frac{1}{p_t} H(t,:)  & \text{ for } i \in \mathcal S_t. \\
\end{array}
\right. 
\] 
Let us illustrate this by a simple example. 


\color{black}
\begin{example}
Let $W=\big[\,m_1\ \ m_2\ \ m_3\,\big]\in\mathbb{R}^{m\times 3}$ with $\kappa(W)>0$,
and let
\[
M=\big[\, \underbrace{m_1, m_1}_{2 \times},\ \underbrace{m_2, m_2, m_2}_{3\times},\ \underbrace{m_3}_{1 \times},\ \tfrac{1}{2}(m_2{+}m_3),\ \tfrac{1}{2}(m_1{+}m_3),\ \tfrac{1}{2}(m_1{+}m_2) \,\big],
\]
with $(p_1,p_2,p_3)=(2,3,1)$ and 
\[
H=\begin{pmatrix}
1 & 1 & 0 & 0 & 0 & 0 & 0   & \tfrac{1}{2} & \tfrac{1}{2} \\
0 & 0 & 1 & 1 & 1 & 0 & \tfrac{1}{2} & 0   & \tfrac{1}{2} \\
0 & 0 & 0 & 0 & 0 & 1 & \tfrac{1}{2} & \tfrac{1}{2} & 0 
\end{pmatrix}.
\]
The optimal $X^*$ for~\eqref{eq:convexmodelSSNMF} in the noiseless case has block
$B_t=\tfrac{1}{p_t} ee^\top$ on the diagonal and $X^*(i,i)=1/p_t$ for $i\in\mathcal{S}_t$, and is given by
$$
X^* =\left(
\begin{array}{cccccc|ccc}
 \nicefrac{1}{2} &  \nicefrac{1}{2} & ~ & ~ & ~ &  ~ &  & \nicefrac{1}{4} & \nicefrac{1}{4}  \\
\nicefrac{1}{2} &  \nicefrac{1}{2} & ~ & ~ & ~& ~ & & \nicefrac{1}{4} &  \nicefrac{1}{4} \\
~  & ~  &  \nicefrac{1}{3} &  \nicefrac{1}{3} & \nicefrac{1}{3} &~ & \nicefrac{1}{6} & &  \nicefrac{1}{6}\\
 ~ & ~ &  \nicefrac{1}{3} &  \nicefrac{1}{3} &  \nicefrac{1}{3} &~  & \nicefrac{1}{6} & & \nicefrac{1}{6} \\
 ~ & ~ &  \nicefrac{1}{3} &  \nicefrac{1}{3} &  \nicefrac{1}{3} & ~ & \nicefrac{1}{6} & & \nicefrac{1}{6} \\
 ~ & ~ & ~ &~ &~ & 1 & \nicefrac{1}{2} & \nicefrac{1}{2} & \\
 \hline
  0 & 0 & 0 &0 &0 & 0 & 0 & 0  & 0 \\
   0 & 0 & 0 &0 &0 & 0 & 0 & 0  & 0 \\ 
      0 & 0 & 0 &0 &0 & 0 & 0 & 0  & 0 
\end{array}
\right), 
$$
where an empty entry means 0. 
\end{example}

\color{black}

\paragraph{Uniqueness of the optimal solution of~\eqref{eq:convexmodelSSNMF}} 
An interesting property of the convex SSNMF model~\eqref{eq:convexmodelSSNMF} is that the optimal solutions share the same diagonal entries in the same block, 
because the objective is strongly convex in these variables. 
Moreover, if the convex hull of the columns of $W$ is a simplex\footnote{That is, the dimension of the affine hull of the columns of $W$ is $r-1$, which is for example the case when $\rank(W) = r$. This is the typical case in practice since $m \gg r$.}, 
then the optimal solution of~\eqref{eq:convexmodelSSNMF} is unique since the matrix $H$ is unique in the decomposition $M = WH$.  
Non-uniqueness of the off-diagonal entries may happen when the columns of $W$ live in an affine space of dimension strictly smaller than $r-1$ (e.g., the columns of $W$ are the vertices of a square in the plane) in which case $H$ might not be unique, e.g., when there are data points in the interior of the convex hull of the columns of $W$. 

This uniqueness property does not hold for the convex SNMF model~\eqref{eq:convsep} where the blocks $B_t$ can be any column stochastic matrices: the objective value, $e^\top \diag(X)$, will be equal to $r$. This is why the postprocessing of the solution of~\eqref{eq:convsep} is crucial in the presence of duplicated (or near-duplicated) columns of $W$~\cite{recht2012factoring, gillis2013robustness, gillis2014robust, mizutani2025endmember}.

\subsection{Recovery Guarantees}\label{subsec:reco_guar}

Let us now provide robustness guarantees of the optimal solution of~\eqref{eq:convexmodelSSNMF} to solve SSNMF, akin to Theorem~\ref{th:robustsepNMF} for SNMF. 
We rely on Assumption~\ref{assum:SSNMF}, so that 
\[
M(:,j) = W(:,t) + \revise{N(:,j)} \text{ for } j \in \mathcal S_t, 
\] 
where $p_t = | S_t | \geq 1$. 
Let us denote $\mathcal S = \cup_{t=1}^r \mathcal S_t$. 
As for SNMF, we assume there exists $\beta < 1$ such that 
\[
M(:,j) = WH(:,j) + \revise{N(:,j)}  
\text{ for } j \notin \mathcal S, 
\text{ where } \max_{t, j \notin \mathcal S} H(t,j) \leq \beta < 1. 
\] 
In other terms, defining $H' = H(:,\mathcal J)$ where $\mathcal J = \{1,\dots,n\} \backslash \mathcal S$, we assume  
$\beta = \max_{t,j} H'(t,j) < 1$. 
Note that $\beta = 1$, that is, $H'(t,j) = 1$ for some $(t,j)$ with $j \notin \mathcal S$, 
would imply that the index $j$ can be added to the set $\mathcal S_t$. Hence, in SSNMF, we can assume w.lo.g.\ that $\beta < 1$.

The robustness proof will follow the same steps as the ones of \cite[Theorem~2]{gillis2014robust} which treats the special case $p_t = 1$ for all $t$. 
The steps are the following: 
\begin{itemize}
    \item Lemma~\ref{lem:sumdiagSt}: Lower bound the diagonal entries of $X$ corresponding to $j \in \mathcal S$. 

    \item Lemma~\ref{lem:upperboundX}: Upper bound the diagonal entries of $X$ corresponding to $j \notin \mathcal S$. 

    \item Theorem~\ref{th:robustSSNMFv1}: Conclude that the diagonal entries of $X$ corresponding to $j \in \mathcal S$ will be larger than the ones  corresponding to  $j \notin \mathcal S$, and hence selecting the largest diagonal entries of $X$ allows us to identify the columns of $M$ corresponding to the columns of $W$. 
 Cluster these columns into $r$ groups. 
    
\end{itemize}


To start the robustness proof, we will first recall the results from  \cite[Lemma 15]{gillis2014robust}. 

\begin{lemma}\cite{gillis2014robust} \label{lem:GRlem15}
Let $M = WH+N$  satisfy Assumption~\ref{assum:stocNMF}, where 
$\|N(:,j)\|_1 \leq \epsilon$ for all $j$. Suppose $X$ is a feasible solution of \eqref{eq:convexmodelSSNMF} and let $M_0=WH$.  Then for all $j$,
\begin{equation}
\|X(:,j)\|_1\leq \frac{1+3\epsilon}{1-\epsilon}\quad \text{and} \quad \|M_0(:,j)-M_0X(:,j)\|_1\leq \frac{4\epsilon}{1-\epsilon}.
\end{equation}   
\end{lemma}

Let us lower bound the diagonal entries corresponding to $\mathcal S$ of any feasible solution $X$. 

{\color{black} \begin{lemma} \label{lem:sumdiagSt}
Let $M = WH+N$ satisfy Assumption~\ref{assum:SSNMF}, 
    where \revise{$\kappa := \kappa(W) > 0$}, $\|N(:,j)\|_1 \leq \epsilon$ for all $j$, and 
    $\max_{k,j \notin \mathcal S} H'(k,j) = \beta < 1$. 
    Then any feasible solution $X$ to the convex model~\eqref{eq:convexmodelSSNMF} satisfies, for all $t =1,2,\dots,r$, 
\[
\sum_{j \in \mathcal S_t} X(j,j) 
\geq 
1 - \frac{4\epsilon(1+\kappa\beta)}{ \kappa (1-\beta) (1-\epsilon)}, 
\]
and 
\[
e^\top X(\mathcal S_t,j) \geq 1 - \frac{4\epsilon(1+\kappa\beta)}{ \kappa (1-\beta) (1-\epsilon)} \; \; \text{ for all } j \in \mathcal S_t.   
\]
\end{lemma}
\begin{proof}
From Lemma \ref{lem:GRlem15}, $M_0=WH$, for $j\in \mathcal S_t$, we have 
\begin{eqnarray*}
\frac{4\epsilon}{1-\epsilon}&\geq&  \left\|M_0(:,j)-M_0X(:,j)\right\|_1
= \left\|W(:,t)-WHX(:,j)\right\|_1\\
&=& \left\|W(:,t)-W(:,t)H(t,:)X(:,j)-\sum_{\ell\neq t}W(:,\ell)H(\ell,:)X(:,j)\right\|_1 
\end{eqnarray*}
Let  $a_\ell=H(\ell,:)X(:,j)$, then $ \sum\limits_{\ell\neq t}W(:,\ell)H(\ell,:)X(:,j)=\sum\limits_{\ell\neq t}a_\ell W(:,\ell)$,

then 
\begin{eqnarray*}
\frac{4\epsilon}{1-\epsilon}&\geq& 
 \left\|W(:,t)-W(:,t)H(t,:)X(:,j)-\sum\limits_{\ell\neq t}a_\ell W(:,\ell)\right\|_1\\
&=&\left\|W(:,t)-W(:,t)H(t,\mathcal{S}_t)X(\mathcal{S}_t,j)-W(:,t)H(t,\mathcal{S} \backslash \mathcal{S}_t)X(\mathcal{S} \backslash \mathcal{S}_t,j)-\sum\limits_{\ell\neq t}a_\ell W(:,\ell)\right\|_1\\
&=&\left\|W(:,t)- \sum_{i\in \mathcal{S}_t}X(i,j) W(:,t)-W(:,t)H(t,\mathcal{S} \backslash \mathcal{S}_t)X(\mathcal{S} \backslash \mathcal{S}_t,j)-\sum\limits_{\ell\neq t}a_\ell W(:,\ell)\right\|_1\\
&\geq &\left\|W(:,t)-\sum_{i\in \mathcal{S}_t}X(i,j)W(:,t)-W(:,t)\left(\sum^n_{i=1}X(i,j)-\sum_{i\in \mathcal{S}_t}X(i,j)\right)\beta-\sum\limits_{\ell\neq t}a_\ell W(:,\ell)\right\|_1\\
&\geq &\left\|W(:,t)-(1-\beta)\sum_{i\in \mathcal{S}_t}X(i,j)W(:,t)-\beta\sum^n_{i=1}X(i,j) W(:,t)-\sum\limits_{\ell\neq t}a_\ell W(:,\ell)\right\|_1\\
&\geq &\left \|W(:,t)-(1-\beta)\sum_{i\in \mathcal{S}_t}X(i,j)W(:,t)-\frac{1+3\epsilon}{1-\epsilon}\beta W(:,t)-\sum\limits_{\ell\neq t}a_\ell W(:,\ell)\right\|_1\\
&=&\left\|\left(1-(1-\beta)\sum_{i\in \mathcal{S}_t}X(i,j)-\frac{1+3\epsilon}{1-\epsilon}\beta\right)W(:t)-\sum\limits_{\ell\neq t}a_\ell W(:,\ell)\right\|_1. 
\end{eqnarray*}
Let $q=1-(1-\beta)\sum\limits_{i\in \mathcal{S}_t}X(i,j)-\frac{1+3\epsilon}{1-\epsilon}\beta$, 
\begin{eqnarray*}
\frac{4\epsilon}{1-\epsilon}&\geq & q\left\|W(:,t)-\sum_{\ell\neq t}W(:,\ell)\frac{a_\ell}{q}\right\|_1\geq q\kappa\\
&= & \left(1-(1-\beta)\sum_{i\in \mathcal{S}_t}X(i,j)-\frac{1+3\epsilon}{1-\epsilon}\beta \right )\kappa.
\end{eqnarray*}
We can deduce that
$$
(1-\beta) \sum_{i\in \mathcal{S}_t}X(i,j)\geq 1-\frac{1+3\epsilon}{1-\epsilon}\beta -\frac{4\epsilon}{(1-\epsilon)\kappa},
$$
which implies that 
\[
e^\top X(\mathcal S_t,j)=\sum_{i \in \mathcal S_t} X(i,j) 
\geq 
1 - \frac{4\epsilon(1+\kappa\beta)}{ \kappa (1-\beta) (1-\epsilon)}. 
\]
Note that $X(i,i)\geq X(i,j)$ for all $i\in \mathcal{S}_t$, it is easy to get,
\[
\sum_{j\in \mathcal{S}_t}X(j,j)\geq e^\top X(\mathcal S_t,j)\geq 
1 - \frac{4\epsilon(1+\kappa\beta)}{ \kappa (1-\beta) (1-\epsilon)}. 
\]
\end{proof}

}

Let us upper bound the diagonal entries not corresponding to $\mathcal S$ of the  optimal solution $X^*$.  

\begin{lemma} \label{lem:upperboundX}
Let $M = WH+N$ satisfy Assumption~\ref{assum:SSNMF}. 
Let also $X^*$ be an optimal solution to the convex model~\eqref{eq:convexmodelSSNMF}, and assume, for all $t =1,2,\dots,r$ that 
\[
\sum_{j \in \mathcal S_t} X^*(j,j) 
\geq 
\gamma,  
\]
for some $0 < \gamma \leq 1$. Then 
\[
X^*(j,j) \leq 1 - \gamma \text{ for all } j \notin \mathcal S. 
\]
\end{lemma}
\begin{proof}
{\color{black}

Let there exist an  optimal solution $X^*$  to convex model \eqref{eq:convexmodelSSNMF}, $\sum_{j \in \mathcal S_t}  X^*(j,j) \geq \gamma$,  but $ X^*(j,j) > 1 - \gamma \text{ for a } j \notin \mathcal S$.  Now construct an $X$:
\[
X(s,\ell) = 
\left\{ 
\begin{array}{cl}
1-\gamma & \text{ for } s = j; ~~\ell = j,  \\
     \frac{\gamma}{p_t} H(t,j)  & \text{ for } s \in \mathcal S_t,~ t=1,2,\cdots, r;~~ \ell=j, \\ 
   0  & \text{ for } s \notin \mathcal S;~~ \ell=j,\\ 
    X^*(s,\ell)  & \text{ for the others.}  \\
\end{array}
\right. 
\] 
We can deduce that  $0\leq X(s,t)\leq X(s,s)\leq 1$,  for $\forall s,t$. Specifically, from Lemma \ref{lem:SOSsum}, we know that $ X^*(s,s)\geq \frac{\gamma}{p_t}$ for $s\in \mathcal S_t$, that is,  $X(s,s)\geq \frac{\gamma}{p_t}$ for $s\in \mathcal S_t$. Note that $H(t,j)\leq 1$, we have $X(s,j)\leq X(s,s)$ for $s\in \mathcal S_t, t=1,2,\cdots, r$. Let us denote $M_0=WH$ and $\mathcal{J}=\{1,\dots,n\}\setminus(\mathcal{S}\cup\{j\})$, we have
\begin{eqnarray*}
&&\|M-MX(:,j)\|_1\\
&=&\left\|M_0(:,j)+N(:,j)-NX(:,j)-WHX(:,j)\right\|_1\\
&=&  \Big\|M_0(:,j)+N(:,j)-NX(:,j) \\
&&-\Big(WH(:,\mathcal{S})X(\mathcal{S},j)+WH(:,j)X(j,j) +WH(:,\mathcal{J})X(\mathcal{J},j)\Big)\Big\|_1\\
 &=& \Big\|M_0(:,j)-WH(:,j)-\Big(WH(:,\mathcal{S})X(\mathcal{S},j)-WH(:,j)\gamma\Big)+N(:,j)-NX(:,j)\Big\|_1\\
&=&  \Big\|-\Big(WH(:,\mathcal{S})X(\mathcal{S},j)-WH(:,j)\gamma\Big)+N(:,j)\\
&&-N(:,j)X(j,j)-N(:,\mathcal{S})X(\mathcal{S},j)
-N(:,\mathcal{J})X(\mathcal{J},j)\Big\|_1\\
&=& \Big\|-\Big(M_0(:,\mathcal{S})X(\mathcal{S},j)-WH(:,j)\gamma\Big)+N(:,j)-N(:,j)+\gamma N(:,j)- N(:,\mathcal{S})X(\mathcal{S},j)\Big\|_1\\
&=&  \left\|- \sum^r_{t=1}p_t*W(:,t)\frac{\gamma H(t,j)}{p_t}+\gamma M_0(:,j)+\gamma N(:,j)-\gamma  \sum^r_{t=1}\frac{\sum\limits_{\ell\in\mathcal{S}_t}N(:,\ell)}{p_t}H(t,j)\right\|_1\\
&=& \gamma \Big\| N(:,j)-\hat{N}H(:,j)\Big\|_1\leq 2 \gamma \epsilon <2\epsilon, 
\end{eqnarray*}
where $\hat{N}:=\frac{\sum\limits_{\ell\in\mathcal{S}_t}N(:,\ell)}{p_t}$ in the last equation. 
Now we can see that such $X$ is a feasible solution, and  $\|\diag(X)\|^2_2 < \|\diag(X^*)\|^2_2$,
which contradicts the fact that $ X^*$ is an optimal solution. 
}
\end{proof}

The above two lemmas allow us to already get a robustness result, similar to~Theorem~\ref{th:robustsepNMF}. Before stating the result, let us define the notion of \emph{trivial clustering} which was used by Arora et al.~\cite{arora2012computing} to prove robustness of SNMF: a set of data points can be trivially clustered if, for some $\epsilon > 0$, all data points are at distance $\epsilon$ of their centroid, while the centroids are at distance more than $4 \epsilon$ of each other. Hence, all data points at distance less than $2 \epsilon$ must belong to the same cluster, which allows one to easily cluster the data points. 
In our setting, since $\|W(:,i) - W(:,j)\|_1 \geq \kappa$ for $i \neq j$, and $\|N(:,j)\|_1 \leq \epsilon$, the columns of $W$ are at distance $4 \epsilon$ as long as $\epsilon < \kappa/4$. 

\begin{theorem} \label{th:robustSSNMFv1}
    Let $M = WH+N$ satisfy Assumption~\ref{assum:SSNMF} \revise{with $\kappa := \kappa(W) > 0$}. 
    Let also $\|N(:,j)\|_1 \leq \epsilon$ for all~$j$,  
    $\max_{k,j \notin \mathcal S} H'(k,j) = \beta < 1$, 
     $p_M = \max_t p_t$, and $X^*$ be an optimal solution of~\eqref{eq:convexmodelSSNMF}.  
If 
\[
\epsilon < \frac{\kappa (1-\beta)}{5 (p_M+1)(1+\kappa \beta)}, 
\] 
then the set 
\[
\mathcal{K} = \left\{ j \ \Big| \ X^*(j,j) >  \frac{1-\delta}{p_M} \right\},  \quad \text{ where } 
\delta = \frac{4\epsilon(1+\kappa \beta)}{\kappa (1-\beta) (1-\epsilon)},  
\] 
contains indices corresponding only to columns of $W$, with at least one index per column of $W$.  
Performing a trivial clustering of the columns of $X(:,\mathcal{K})$ leads to a solution $\tilde{W}$ such that 
    \[
    \|W - \tilde{W} \Pi\|_1 \leq \epsilon \quad \text{ for some permutation matrix } \Pi.   
    \] 
\end{theorem}
\begin{proof}
    By Lemma~\ref{lem:sumdiagSt}, at least one entry $j \in \mathcal{S}_t$ has diagonal entry larger than $\frac{1-\delta}{p_t}$. 
    By Lemma~\ref{lem:upperboundX}, the diagonal entry $X(j,j)$ for any $j \notin \mathcal{S}_t$ is smaller than $\delta$. Hence, if 
\[
\frac{1-\delta}{p_t}  
\geq 
\frac{1-\delta}{p_M} > \delta, 
\]
   the set $\mathcal{K}$ contains only entries corresponding to columns of $W$, at least one for each column of $W$, hence trivial clustering will provide the result.     

    It suffices to check that the choice of $\epsilon$ is sufficiently small so that $\frac{1-\delta}{p_M} > \delta$, that is, so that 
    $\delta (p_M+1) < 1$. 
    This requires 
    \[
\frac{4\epsilon(1+\kappa\beta)}{\kappa (1-\beta) (1-\epsilon)} (p_M+1) < 1 \iff 
\frac{\epsilon}{1-\epsilon} < \frac{\kappa (1-\beta)}{4 (p_M+1)(1+\kappa\beta)}. 
    \] 
We can further simplify the bound, since $\kappa \leq 1$ and $p_M \geq 1$, letting $\epsilon \leq 1/10$, so $1-\epsilon \geq 9/10$, then $\frac{\epsilon}{1-\epsilon} \leq \frac{\epsilon}{9/10}$, 
leading to the requirement 
$$
\frac{\epsilon}{9/10}< \frac{\kappa (1-\beta)}{4 (p_M+1)(1+\kappa\beta)} \iff \epsilon<\frac{\kappa (1-\beta)}{\frac{40}{9} (p_M+1)(1+\kappa\beta)}, 
$$
which is implied by $\epsilon< \frac{\kappa (1-\beta)}{5 (p_M+1)(1+\kappa\beta)}$.  
\end{proof}

Theorem~\ref{th:robustSSNMFv1} is not fully satisfying since this does not lead to more robustness than the standard SNMF case; note that it essentially coincides (up to constant factors) to Theorem~\ref{th:robustsepNMF} when $p_t = 1$ for all $t$.  
This is because we do not leverage the presence of multiple points close to each column of $W$. How can we ensure that all columns of $X$ corresponding to columns of $W$ belong to $\mathcal{K}$? This can be proved by leveraging the objective function, $\| \diag(X)\|_2^2$. 

In the following, we prove that every entry in $\mathcal S_t$ can be lower bounded, not only their sum as in Lemma~\ref{lem:sumdiagSt}. 
Hence we will be able to identify an index set, $\mathcal{K}$, corresponding to  $p_t$ copies of  $W(:,t)$. The aggregation step will therefore allow us to reduce the noise when estimating $W$.

In the following lemma, we show how the objective function would increase compared to Lemma~\ref{lem:SOSsum} if one variable is imposed to be smaller than the others. This will allow us, in Lemma~\ref{lem:allentriesXjj}, to lower bound all entries $X(j,j)$ for $j \in \mathcal S_t$, 
not only their sum (Lemma~\ref{lem:sumdiagSt}).  
\begin{lemma} \label{lem:SOSsum2} 
Let $1 < p \in \mathbb{N}$, 
$\alpha > 0$, 
$\beta < \frac{\alpha}{p}$ 
and 
$j \in \{1,2,\dots,p\}$. 
The optimal solution to the optimization problem 
\[
\min_{x \in \mathbb{R}^p} 
\sum_i x_i^2 
\text{ such that } 
\sum_i x_i \geq \alpha 
\text{ and } x_j \leq \beta,  
\]
is given by $x_j^* = \beta$ and $x_i^* = \frac{\alpha-\beta}{p-1}$ for all $i \neq j$, with objective $f^* = \beta^2 + \frac{(\alpha-\beta)^2}{p-1}$. 

Hence the second constraints, $x_j \leq \beta$, increases the objective by 
\[
\beta^2 + \frac{(\alpha-\beta)^2}{p-1} - \frac{\alpha^2}{p} =
\frac{(\alpha - p \beta)^2}{p(p-1)} \geq 0. 
\]
\end{lemma}
\begin{proof} The optimality conditions are  $\frac{\partial}{\partial  x} \big(\sum\limits^p_{i=1} x_i^2 +\lambda (\alpha-\sum\limits^p_{i=1} x_i)+\mu (x_j-\beta)\big)=0$,  which gives $x = \lambda e - \mu e_j$, 
where $\lambda \geq 0$ is the Lagrange multiplier of the constraint $\sum_i x_i \geq \alpha$ and $\mu\geq 0 $ of $x_j \leq \beta$. Since $\beta < \frac{\alpha}{p}$, the second constraint must be active (otherwise $\mu = 0$ which leads to a contradiction since all entries of $x^*$ will be equal to one another) hence $x_j^* = \beta$ implying $x_i^* = (\alpha-\beta)/(p-1)$. 
\end{proof}

We can now combine Lemmas~\ref{lem:sumdiagSt} and \ref{lem:SOSsum2} to lower bound individual diagonal entries of $X$ corresponding to $j \in \mathcal S_t$. 
Before doing so, let us define the quantity, 
\[
\tilde r 
\quad := \quad 
\sum_{t=1}^r \frac{1}{p_t} 
\quad \leq \quad 
r, 
\]
which we call the effective rank. 
The value $\tilde r$ is an upper bound of the optimal value of~\eqref{eq:convexmodelSSNMF}, as shown by the following lemma. 
\begin{lemma} \label{lem:feasiblesol} 
Let $M = WH+N$ satisfy Assumption~\ref{assum:SSNMF}. 
There is a feasible solution $X$ to~\eqref{eq:convexmodelSSNMF} with objective function value $\tilde r = \sum_{t = 1}^r \frac{1}{p_t}$. 
\end{lemma}
\begin{proof} 
Take 
\[
X(i,:) = 
\left\{ 
\begin{array}{cl}
   0  & \text{ for } i \notin \mathcal S,  \\
   \frac{1}{p_t} H(t,:)  & \text{ for } i \in \mathcal S_t. \\
\end{array}
\right. 
\]
The centroid $\tilde W(:,t) = \frac{1}{p_t} \sum_{j \in \mathcal S_t} M(:,j)$ is such that $\|W-\tilde W\|_1 \leq \epsilon$, since $\| M(:,j) - W(:,t) \|_1 \leq \epsilon$ for $j \in \mathcal S_t$,  
and hence $\|M - MX\|_1 \leq 2 \epsilon$ since $H$ is column stochastic.  

This solution $X$ gives $X(j,j) = 1/p_t$ for all $j \in \mathcal S_t$, since $H(:,j) = e_t$ for $j \in \mathcal S_t$. Hence the objective function value is equal to $\sum_{t=1}^r \frac{1}{p_t}$. 
Note also that $X(i,j) = 1/p_t$ for all $i,j \in \mathcal{S}_t$ and
$X(i,j) = 0$ for all $i \in \mathcal{S}_t$ and $j \in \mathcal{S}_{t'}$ for $t \neq t'$. 
\end{proof}

In practice, the effective rank is a measure of how far SSNMF is from SNMF. We have  $\tilde r = r \iff p_t =1$ for all $t$, while $\tilde r \rightarrow 0$ when $p_t \rightarrow \infty$ while $r$ is fixed. 
For example, in a hyperspectral image, it is customary to have $p_t \approx 100$ for all $t$, while $r \approx 10$, in which case $\tilde r \approx \frac{1}{10} \ll r \approx 10$.   

\begin{lemma} \label{lem:allentriesXjj} 
Let $M = WH+N$ satisfy Assumption~\ref{assum:SSNMF}, 
    where $\|N(:,j)\|_1 \leq \epsilon$ for all $j$, and 
    $\max_{k,j \notin \mathcal S} H'(k,j) = \beta < 1$. 
    Then the optimal solution to the convex model~\eqref{eq:convexmodelSSNMF} $X^*$ satisfies 
\[
X^*(j,j) \geq \frac{1}{p_t} - \frac{\delta}{p_t} - \sqrt{\delta \tilde r} 
\quad 
\text{ for all } j \in \mathcal S_t, 
\]
where  $\delta = \frac{4\epsilon(1+\kappa \beta)}{ \kappa (1-\beta) (1-\epsilon)}$. 
\end{lemma}
\begin{proof}   
By Lemma~\ref{lem:sumdiagSt}, $\sum_{j \in \mathcal{S}_t} X(j,j) \geq 1- \delta$. Hence, by Lemma~\ref{lem:SOSsum}, 
$\sum_{j \in \mathcal{S}_t} X(j,j)^2 \geq \frac{1-\delta}{p_t}$. 
Combined with Lemma~\ref{lem:feasiblesol}, we have 
\begin{equation} \label{eq:sandwichobj}
\sum_t \frac{1-\delta}{p_t} 
=
(1-\delta) \tilde r 
\;  \leq \; 
\|\diag(X^*)\|_2^2 
\;  \leq \;  \tilde r = 
\sum_t \frac{1}{p_t}. 
\end{equation}

For simplicity, let us focus on one $t \in \{1,2,\dots,r\}$ and  denote $p = p_t$, $x \ \revise{\in} \ \mathbb{R}^{p}$ the vector containing $X^*(j,j)$ for $j \in S_t$. 
By Lemma~\ref{lem:sumdiagSt}, we have 
\[
\sum_{i=1}^{p_t} x_i  
\geq  \alpha = 1 - \delta.   
\]
Now assume $x_j \leq \beta < \frac{\alpha}{p}$ for some $j$, by Lemma~\ref{lem:SOSsum2}, it increases the objective by $\frac{(\alpha - p \beta)^2}{p(p-1)}$, but it cannot increase it by more than $\delta \tilde r$, otherwise this will lead to a contradiction since there is a feasible solution with objective $\tilde r$; see~\eqref{eq:sandwichobj}.  
Hence 
\[
\frac{(\alpha - p \beta)^2}{p^2} 
\leq 
\frac{(\alpha - p \beta)^2}{p(p-1)} 
\leq  
\delta \tilde r, 
\] 
Since $\beta < \frac{\alpha}{p}$, we have 
$\alpha - p \beta \leq  p \sqrt{\tilde r \delta}$, 
which leads to 
\[
\beta \geq \frac{\alpha}{p} - \sqrt{\delta \tilde r} 
= \frac{1}{p} - \frac{\delta}{p} - \sqrt{\delta \tilde r}. 
\]  
\end{proof}

We can now show under which conditions the set $\mathcal S$ can be identified.  
\begin{theorem} \label{th:mainth}
Let $M = WH+N$ satisfy Assumption~\ref{assum:SSNMF} \revise{with $\kappa := \kappa(W) > 0$}.  
    Let also $\|N(:,j)\|_1 \leq \epsilon$ for all~$j$,  
    $\max_{k,j \notin \mathcal S} H'(k,j) = \beta < 1$, 
     $p_M = \max_t p_t$, and $X^*$ be an optimal solution of~\eqref{eq:convexmodelSSNMF}. 
If 
\[
\epsilon < 
\frac{\kappa (1-\beta)}{18 p_M^2 \tilde r}  
\] 
then 
\[
\mathcal{K} \quad = \quad  \left\{ j \ \Big| \ 
X^*(j,j) 
>  
\frac{1-\delta}{p_M} 
- \sqrt{\delta \tilde r} \right\} 
\quad = \quad 
\mathcal S, 
\] 
where $\delta = \frac{4\epsilon(1+\kappa\beta)}{ \kappa (1-\beta) (1-\epsilon)}$. 
Performing a trivial clustering of the columns of $X(:,\mathcal{K})$ leads to a solution $\tilde{W}$ such that     
    $\|W(:,j) - \tilde{W} \Pi\|_1 \leq \epsilon$ for some permutation matrix $\Pi$.    

\end{theorem}
\begin{proof}
 All diagonal entries corresponding to $\mathcal S_t$ have values 
    larger than $\frac{1 - \delta}{p_t} - \sqrt{\delta \tilde r}$ (Lemma~\ref{lem:allentriesXjj}). 
    All others have value smaller than $\delta$ (Lemma~\ref{lem:upperboundX}). 
    Hence it remains to choose $\epsilon$ sufficiently small so that 
    \[
\frac{1 - \delta}{p_t} - \sqrt{\delta \tilde r}  > \delta. 
    \] 
The above inequality can be rewritten as 
    \[
 (p_t + 1) \delta  + p_t \sqrt{\delta \tilde r} < 1 , 
    \] 
    which is implied by the following two conditions: 
    \begin{itemize}
            \item $(p_t + 1) \delta < 1/2$. This is a condition similar to Theorem~\ref{th:robustSSNMFv1} (which required $(p_t + 1) \delta < 1$), and any $\epsilon < \frac{\kappa (1-\beta)}{9 (p_M+1)}$ provides the desired result. 
            This is satisfied since $p_M \tilde{r} = \sum_{t=1}^r p_M/p_t \geq  r p_M \geq r \geq 1$, hence 
            $9 (p_M+1) \leq 18 p_M  \leq 18 p_M^2 \tilde r$. 

        \item $p_t \sqrt{\delta \tilde r} < 1/2$. This is implied by 
        \[
 \delta = 4 \frac{\epsilon}{\kappa (1-\beta) (1-\epsilon)} 
 < \frac{1}{4 p_M^2 \tilde{r}}. 
        \]
Using a similar argument as in Theorem~\ref{th:robustSSNMFv1}, this inequality is satisfied for $\epsilon < \frac{\kappa (1-\beta)}{18 p_M^2 \tilde r}$.  
    \end{itemize} 
\end{proof}

\subsection{Discussion on the Robustness Guarantees} 

The robustness of Theorem~\ref{th:mainth} is not stronger than in the case $p_t = 1$. The reason is that we want to identify more columns, so the problem is intrinsically harder, while we make no assumption on the noise beyond boundedness, so averaging is not guaranteed to decrease the estimation error. The reason is that the noise could be adversarial, e.g., $N(:,j) = N(:,i)$ for all $i,j \in \mathcal{S}_t$ in which case averaging columns of $M(:,\mathcal S_t)$ does not reduce the noise.

\paragraph{Reducing the variance to better estimate $W$}

However, if the entries of $N$ are independently distributed,  
averaging the columns of \mbox{$\{M(:,j)\}_{j \in \mathcal S_t}$} will reduce the noise, \emph{since the variance will be reduced by a factor $p_t$}, in which case SSNMF will outperform SNMF. 
For example, in the case of Gaussian noise, we have the following result, \revise{which follows from the fact that the standard deviation of the average of $n$ independent Gaussian distributed variables with standard deviation $\sigma$ is a Gaussian distributed variable with standard deviation $\sigma/\sqrt{n}$.}
\begin{corollary}
Let $M = WH + N$ satisfy the assumptions of Theorem~\ref{th:mainth} and the entries of $N$ follow an i.i.d.\ Gaussian distribution. 
Let us denote $\tilde W(:,t)$ the average of the columns of 
$M(:,\mathcal S_t)$ extracted by solving the convex SSNMF model~\eqref{eq:convexmodelSSNMF}. 
We have, for all $t$ and $j \in \mathcal S_t$, 
\[
\mathbb E \left\|W(:,t) - \tilde W (:,t) \right\|_2 
\;  = \; 
\frac{1}{\sqrt{p_t}} 
\mathbb E \Big\|W(:,t) - M(:,j) \Big\|_2 , 
\] 
 where $\mathbb E$ denotes the expectation. 
\end{corollary}

\paragraph{Robustness to outliers} 

Using the median\footnote{\revise{In this paper, for simplicity, we use the component-wise median of a set of vectors. 
Other choices would be possible, e.g., the geometric median that minimizes the sum of the $\ell_2$ norm between the median and the set of vectors.}} instead of the mean would make the algorithm robust against misclassifications by the clustering procedure, 
as long as at least half the data points in each cluster correspond to correct vertices. 
In particular, this allows us to deal with outliers, as shown in the following corollary. In fact, if using an aggregation technique robust to outliers, such as the median, our convex model can deal with $\min_t p_t$ outliers without any particular pre- or post-processing. 
\begin{corollary}
Let $M = [WH+N, B] \Pi$, where $WH+N$ is as in Theorem~\ref{th:mainth}, $B \in \mathbb{R}^{m \times \ell}$ are outliers with $\ell < \min_t p_t$, 
and $\Pi$ is a permutation matrix.  
Replacing $\kappa$ in Theorem~\ref{th:mainth} by $\kappa' = \kappa([B, W]) > 0$, will imply that $\mathcal{K}$ identifies the columns of $W$ ($p_t$ copies for $W(:,t)$) and $B$ (1 copy for each). 
Performing any clustering of the columns of $M(:,\mathcal{K})$ that ensures that the columns at a distance less than $2 \epsilon$ are in the same cluster, and defining $\tilde W$ as the median of each cluster guarantees $\|W - \tilde{W} \Pi \|_1 \leq \epsilon$ for some permutation $\Pi'$, since $\min_t p_t$ is larger than the number of outliers. 
\end{corollary} 
\begin{proof} 
This follows directly from Theorem~\ref{th:mainth}. One can think of the columns of $B$ as $\ell$ vertices that are not used by any data point.  

\end{proof}


We will illustrate this robustness with the numerical experiments; see Section~\ref{sec:outliers}, injecting outliers and observing that median aggregation preserves correct recovery whenever at least half of the points in each cluster are inliers.

\section{A Two-step Algorithmic Pipeline}\label{sec:Algo}

A key aspect when using convex models is to properly postprocess the optimal solution $X^\star$ of a convex model; 
see~\cite{recht2012factoring, gillis2013robustness, gillis2014robust, mizutani2025endmember} 
and the references therein. 
Our proposed algorithm follows a two-step pipeline. 

\paragraph{Step 1 - Compute \(X^*\).}  

In practice, it is hard to estimate $\epsilon$ in the convex model for SSNMF~\eqref{eq:convexmodelSSNMF}. Hence, one typically resorts to a penalty approach and solves instead
\begin{equation}
  \label{eq:modelSSNMFpenlaty}
  \min_{X \in \Omega} \ \lVert M - M X \rVert_F^2 \;+\; \mu \, \lVert \operatorname{diag}(X) \rVert_2^2 ,
\end{equation}
for some penalty parameter $\mu$. To be able to use the model when $H$ is not column stochastic, we use the same feasible set as in~\cite{gillis2014robust}:
\[
  \Omega := \bigl\{\, X \in \mathbb{R}^{n\times n}_{+} \ \bigm|\ X(i,i) \le 1,\ \ w_i\, X(i,j) \le w_j\, X(j,j)\ \ \forall\, i,j \,\bigr\},
\]
where the vector $w \in \mathbb{R}^{n}_{+}$ contains the column $\ell_{1}$ norms of $M$, that is, $w_j=\lVert M(:,j)\rVert_{1}$ for all~$j$. 
\revise{This allows one to avoid normalization of the input matrix that can amplify the noise, in particular of data points with small norms, e.g., background pixels in hyperspectral images, or short documents in word-count matrices; see the discussion in~\cite{gillis2014robust}.} 

\color{black}
Instead of using the norm $\|\cdot\|_1$, we adopt the Frobenius norm in the data-fidelity term. This choice is motivated by several reasons: 
\begin{itemize}
    \item It is a standard choice in the NMF community, and it corresponds to assuming i.i.d.\ additive Gaussian noise. This was shown to perform well for SNMF in~\cite{gillis2018fast}. 

    \item It makes the optimization problem smooth, and hence amenable to fast first-order methods; 
    see Algorithm~\ref{algo:step1}. In fact, the gradient of the objective is $L$-smooth with \mbox{$L = 2\bigl(\lVert M \rVert_{2}^{2} + \mu\bigr)$}. 
    \revise{Algorithm~\ref{algo:step1} is directly adaptable for any Lipschitz smooth function. However, for cost functions that are not Lipschitz smooth, such as the $\ell_1$ norm or the Kullback-Leibler divergence, one would need to resort to other optimization strategies (e.g., subgradient or mirror descent methods). This is a topic for further research.} 

\end{itemize} 
Note also that~\eqref{eq:modelSSNMFpenlaty} is the Lagrangian relaxation of 
\[ 
\min_{X \in \Omega}\ \lVert \operatorname{diag}(X) \rVert_{2}^{2}
  \ \ \text{ such that }\ \ \lVert M - M X \rVert_{F} \le \epsilon . 
\]  

\color{black}
To solve~\eqref{eq:modelSSNMFpenlaty}, we adapt the fast gradient method of Nesterov~\cite{nes83} developed for the penalized convex SNMF model~\eqref{modelsepNMF} proposed in~\cite{gillis2018fast}; see Algorithm~\ref{algo:step1}. 
\revise{For simplicity, we run Algorithm~\ref{algo:step1} for a fixed number of iterations, 1000 by default. 
Using early stopping would be useful to reduce the computational cost,  
motivated by the facts that gradient descent typically makes most progress in the early steps and that we do not need a high-accuracy solution. This is a tuning aspect of the algorithm that is not explored in this paper.}

\begin{algorithm}[ht!]
\caption{Fast gradient method for~\eqref{eq:modelSSNMFpenlaty} \label{algo:step1}}
\begin{algorithmic}[1]
\REQUIRE \(M\in \mathbb{R}^{m\times n}_+\), a penalty parameter \(\mu\), maximum number of iterations \texttt{maxiter}, extrapolation parameter $\alpha_0 \in (0,1)$. 
\ENSURE Approximate solution \(X\) to~(\ref{eq:modelSSNMFpenlaty})
\STATE \text{\% Initialization}
\STATE  $\alpha_0 \leftarrow 0.05$; 
$X \leftarrow 0_{n,n}$; 
$Y \leftarrow X$; 
\STATE Compute the Lipschitz constant $L \leftarrow 2\sigma_{\max}^2(M) + 2 \mu$ 
\FOR{\(k=1:\) maxiter} 
\STATE \(X_p \leftarrow X\) 
\STATE \text{\% Gradient computation}
\STATE $\nabla F(Y) \leftarrow 2 M^\top (M Y - M) \;+\; 2\,\mathrm{Diag}(\operatorname{diag}(Y))\,\mu$ 
\STATE \text{\% Projected Gradient Step with projection on $\Omega$; see Sec.~III-D in \cite{gillis2018fast}}
\STATE $X \leftarrow \mathcal{P}_{\Omega}\!\Big(Y - \frac{1}{L}\nabla F(Y)\Big)$
\STATE \text{\% Extrapolation step}
\STATE \(Y \leftarrow X + \beta_k (X - X_p)\), where \(\beta_k=\frac{\alpha_{k-1}(1-\alpha_{k-1})}{\alpha^2_{k-1} + \alpha_k}\) such that \(\alpha_k \geq 0\) and \(\alpha_k^2 = (1-\alpha_k) \alpha^2_{k-1}\)
\ENDFOR
\end{algorithmic}
\end{algorithm}

\paragraph{Step 2 - Compute the NMF factors} 

We propose to select the indices of the columns of $M$ not according to the diagonal entries of $X^*$ (as in Theorems~\ref{th:robustSSNMFv1} and~\ref{th:mainth}), but 
according to the $\ell_1$ norms of its rows. We observed that it performs better in practice. The reason is that it leverages all the entries of $X^*(\mathcal S, :)$, hence leverages the block structure of $X^*$. 
For $j \in \mathcal S_t$, we have a lower bound on $X(j,j)$. Interestingly, we can actually prove that $X(i,j)$ satisfies a similar lower bound for $i \in \mathcal S_t$. Hence $\|X(j,S_t)\|_1$ is proportional to $p_t X(j,j)$. 
Moreover, $X(j,i)$ for $i \notin \mathcal S_t$ will be proportional to $H(t,i)/p_t$, while the other entries of $X$ will be upper bounded by $\delta$. 

In practice, it is not easy to select a threshold for selecting the indices corresponding to these rows. 
\revise{One reason is that, in practice, the noise level is typically high and hence the theoretical bounds derived in this paper might not be applicable.} 
For simplicity, we propose to choose a fixed number $p$ of indices to keep. Ideally $p = \sum_{t=1}^r p_t$. Although this quantity is unknown, this already offers a significant advantage compared to existing SSNMF algorithms that need to guess $p_t$ for all~$t$.

Once the index set $\mathcal K$ corresponding to the $p$ largest  values \revise{among} $\{\|X(i,:)\|_1\}_{i=1}^n$ have been selected, we use \textit{spectral clustering}, a standard graph clustering method that utilizes eigenvectors of a similarity matrix~\cite{von2007tutorial}. 
\revise{Spectral clustering is a standard in graph clustering, and shows good performance and robustness.}
Since spectral clustering requires a symmetric similarity matrix, we apply it to the {symmetrized} matrix, 
$S = 1/2 \left(X(\mathcal{K},\mathcal{K}) + X(\mathcal{K},\mathcal{K})^T \right)$. 
\revise{Spectral clustering is applied on a $p$-by-$p$ matrix, where $p \ll n$, whose main cost is to compute $r$ leading eigenvectors, with  cost $\mathcal{O}(p^2 r)$ operations with iterative methods. 
Hence its cost is significantly smaller than that of solving our convex optimization problem with Algorithm~\ref{algo:step1}, in $\mathcal{O}(n^3)$ operations.}

This overall procedure is summarized in Algorithm~\ref{algo:post}.

\revise{
\begin{remark}[Other choices for the postprocessing] 
Our algorithm available online allows the user to choose to cluster the extracted data points, $M(:,\mathcal K)$, via k-means. We observed that this variant performs on average similarly to spectral clustering. Note that spectral clustering also relies on k-means, and we used the default MATLAB implementation that uses k-means++ as initialization, which has theoretical guarantees~\cite{arthur2007kmeanspp}. 
\end{remark}
}

\begin{algorithm}[ht]
\caption{CSSNMF: SSNMF algorithm based on the convex model~\eqref{eq:modelSSNMFpenlaty} \label{algo:post}}
\begin{algorithmic}[1]

\REQUIRE A smooth separable matrix \(M\in \mathbb{R}^{m\times n}_+\), 
parameter $\mu > 0$, 
number \(p\) of columns to extract, 
cluster number \(r\). 

\ENSURE A set \(\mathcal{K}\) of column indices in \(r\) groups, factor matrices \(W\in \mathbb{R}^{m\times r}\) and \(H\in \mathbb{R}^{r\times n}_+\) such that $M \approx WH$. 

\STATE Run Algorithm~\ref{algo:step1} to obtain an approximation solution, $X$, to the penalized convex SSNMF model~\eqref{eq:modelSSNMFpenlaty}  

\STATE  Construct \(\mathcal{K}\) from \(X\) by selecting the $p$ rows with 
largest $\ell_1$ norm. 

\STATE  Perform spectral clustering on \(S = \frac{1}{2} 
(X(\mathcal{K},\mathcal{K}) + X(\mathcal{K},\mathcal{K})^T)\) to obtain \(r\) clusters \(\{\mathcal{K}_t\}_{t=1}^r\). 
\STATE Compute \(W(:,t) = \frac{1}{|\mathcal{K}_t|} \sum_{j \in \mathcal{K}_t} M(:,j)\), for \(t = 1,\dots, r\).
\STATE Solve \(H = \arg\min_{Z\in \mathbb{R}^{r\times n}_+}\|M-WZ\|^2_F\).
\end{algorithmic}
\end{algorithm}

\paragraph{Handling the unknown number of selected indices} 

In practice, the exact value of \(p\) is unknown. Another practical alternative is to select rows dynamically using a threshold \(\delta\): Select \(\mathcal{K} = \{ j \ | \ \|X(j,:)\|_1 \geq \delta \} \) where \( \delta > 0 \) is a user-defined threshold, and such that $|\mathcal{K}|$ is at least larger than $r$.


\paragraph{Choice of the parameter $\mu$}



Instead of fixing $\mu$, we adapt it online by steering statistic $T(X)$ toward target $T^\star$, with two possible choices, diagonal or residual: 
$$T(X) \in \left\{ \|\operatorname{diag}(X)\|_2 \text{(diagonal)}, \ \|M-MX\|_F \text{(residual)} \right\},$$
where $T^\star = \tau$ (diagonal) or $\rho$ (residual). After each restart (or every $k$ iterations), update:
$$\mu_{t+1} = (1 \pm \sigma_t) \mu_t, \quad \sigma_t > 0,$$
choosing the sign to move $T(X_{t+1})$ toward $T^\star$ based on known monotonicity. Halve $\sigma_t$ when direction reverses to prevent oscillation. This lightweight proportional controller finds effective $\mu$ without line searches and works across datasets.

\emph{Default policy used in our experiments.}
We target the diagonal mass via a simple surrogate: we set a nominal target $\tau \approx r/2+1$ for $\operatorname{tr}(X)$ and, at natural restart points, adjust $\mu$ multiplicatively according to the rule above.
A coarse warm start provides the initial $\mu$, and the closed loop rapidly drives it to an effective regime.
Since $X\ge 0$ in $\Omega$, both $\operatorname{tr}(X)$ and $\|\operatorname{diag}(X)\|_2$ decrease as $\mu$ increases, making $\operatorname{tr}(X)$ a reliable proxy for controlling the diagonal magnitude.


\section{Numerical Experiments} \label{sec:numexp}

In this section, we conduct experiments on synthetic (Section~\ref{sec:synth}) and hyperspectral (Section~\ref{sec:hyper}) datasets to evaluate the performance of the proposed algorithm CSSNMF.
All experiments were run on a MacBook Pro (Apple M4, 24\,GB RAM) using MATLAB.
The full MATLAB code and test scripts that regenerate all results reported in this section are publicly available at \url{https://github.com/vleplat/ConvexSmoothSeparableNMF.git}.

We compare our methods with the following three baselines:
\begin{itemize}
\item  Two algorithms for separable NMF: 

(1) SPA (Successive Projection Algorithm)~\cite{araujo2001successive,ma2014signal,gillis2014fast,barbarino2025robustness} selects the
column with the largest $\ell_2$ norm and projects all columns of $M$ on the orthogonal complement of the extracted column at each step.

(2)  FGNSR \cite{gillis2018fast} is a fast algorithm based on the convex SNMF model determining  $X^*$ with post-process by applying SPA to its rows to identify the columns. 

\item  SSPA (Smooth SPA)~\cite{NADISIC2023174} extends   SPA for solving smooth separable NMF. For SSPA applied on synthetic datasets, let $p_t$ denote the size of class $t$ and define
\[
p_{\min}=\min_t p_t, \qquad \bar p=\frac{1}{r}\sum_{t=1}^r p_t.
\]
We consider three variants by setting the number of proximal latent points (plp) to
\[
\text{SSPA(min)}:\ n_{\mathrm{plp}}=p_{\min},\quad
\text{SSPA(mid)}:\ n_{\mathrm{plp}}=\Big\lfloor \tfrac{p_{\min}+\bar p}{2}\Big\rfloor,\quad
\text{SSPA(mean)}:\ n_{\mathrm{plp}}=\big\lfloor \bar p \big\rceil,
\]
where $\lfloor\cdot\rfloor$ and $\lfloor\cdot\rceil$ denote floor and nearest-integer rounding, respectively.
  
\end{itemize}



\subsection{Synthetic Datasets}
\label{sec:synth}

In this section, we compare the different algorithms on three synthetic scenarios:
(i) random mixtures with Dirichlet coefficients,
(ii) the middle-point experiment with adversarial noise, and
(iii) robustness to outliers (median vs.\ average aggregation). 
We will report the following three quality measures. 
The first two are the same as in~\cite{gillis2014robust}. In all cases, the entries of $W$ are generated uniformly at random in $[0,1]$. 

 
 
\begin{enumerate}
\item \textbf{Accuracy.} Let $J_0$ index the $n_0$ one-hot columns and let $y_j\in\{1,\ldots,r\}$ be the true class of column $j\in J_0$.
For each method, the computed $\widehat H$ infers predicted labels $\hat y_j=\arg\max_t \widehat H(t,j)$, and align labels by an optimal permutation $\Pi\in\mathfrak{S}_r$.
The clustering accuracy is defined by finding the best match between the predicted clusters and the true class labels:
\begin{equation}\label{accuracy}
    \mathrm{Acc} \;=\; \frac{1}{n_0}\,\max_{\Pi\in\mathfrak{S}_r}\ \sum_{j\in J_0}\mathbf{1}\{\hat y_j=\Pi(y_j)\},
\end{equation}
where $\mathbf{1}$ is the indicator function, which returns $1$ if it is correctly matched, and $0$ otherwise. 

\item \textbf{Relative $W$-error.} We define 
\begin{equation}\label{W_error}
d_W \;=\; \min_{\Pi}\ \frac{\|\widehat W \Pi - W\|_F}{\|W\|_F},
\end{equation}
computed after column-wise $\ell_2$ normalization of $W$ and $\widehat W$. 


\item \textbf{Relative approximation error.} It is efined as 
\begin{equation}
\label{error}
\frac{\min_{P \ge 0}\ \|M - W P\|_F}{\|M\|_F},
\end{equation}
where $W = M(:,\mathcal{K})$ for SPA/SSPA/FGNSR. For our algorithm, $W$ is obtained by clustering the selected columns of $M$ and taking the average of the columns belonging to the same class. The matrix $P$ is computed by a coordinate-descent NNLS solver as in~\cite{gillis2012accelerated}.
\end{enumerate}

\subsubsection{Fully randomly generated data (Dirichlet mixtures)}
\label{sec:randnoise}

We generate noisy matrices $M \in \mathbb{R}^{30 \times 100}$ as
\[
M = \max(0,\,M_0 + N), \qquad M_0 = W [H_0,\; H_1],
\]
with the following setting:  
\begin{itemize}

\item The block $H_0 \in \mathbb{R}^{5 \times 50}$ is one-hot (each column is a canonical basis vector) that is constructed such that the columns of $M_0$ are partitioned into $ r=5 $ classes, with each class containing exactly $ p_t = n_0 / r = 10 $ columns for $ t = 1, \dots, r $. Matrix $H_1 \in \mathbb{R}^{5 \times 50}$ is drawn i.i.d.\ from the uniform Dirichlet distribution of parameter $\alpha$.    

\item The noise matrix $N$ is Gaussian with i.i.d.\ entries scaled so that $\|N\|_F = \epsilon \|M_0\|_F$. We apply per-column $\ell_1$ normalization to $M$ so that the sum of each column of $M$ is 1.
 
\end{itemize}

We use 7 different noise levels $\epsilon$ logarithmically spaced in $[10^{-5},
 10^{-0.05}]$ (\texttt{logspace(-5,-0.05,7)} in MATLAB). For each noise level, we generate 20 such matrices and report the average quality measures on Figure \ref{fig:rand_avg}. 
We have the following observations:

\begin{enumerate}[label=(\roman*)]
\item In terms of accuracy, across noise levels, our CSSNMF attains the highest accuracy overall, followed by FGNSR, then the
SSPA variants (with a small and consistent ordering: mean $\ge$ mid $\ge$ min, where $\ge$ means it performs better), 
while SPA typically ranks last.
All methods achieve decent accuracy at low noise, while as noise increases, the performance gap widens significantly.
\item In terms of relative approximation error and the relative $W$-error, the rankings are also similar: our method is the best or tied for best on average, followed closely by FGNSR, with the SSPA variants clustering closely together, while SPA lags behind.
\end{enumerate}

In summary, our CSSNMF exhibits the best overall robustness and performance across all evaluated metrics and noise levels on Dirichlet mixtures, particularly excelling in maintaining low error and high accuracy under significant noise corruption.

 \begin{figure}[t]
  \centering
  \includegraphics[width=0.33\textwidth]{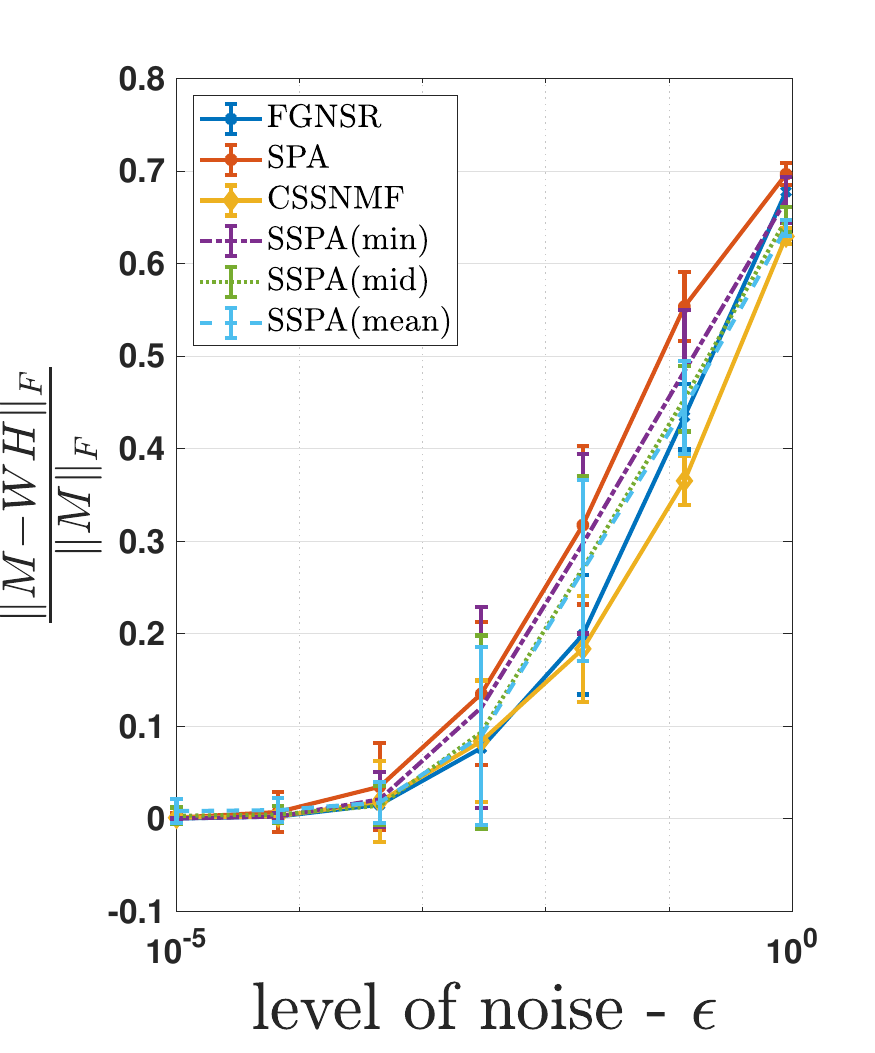}\hfill
  \includegraphics[width=0.33\textwidth]{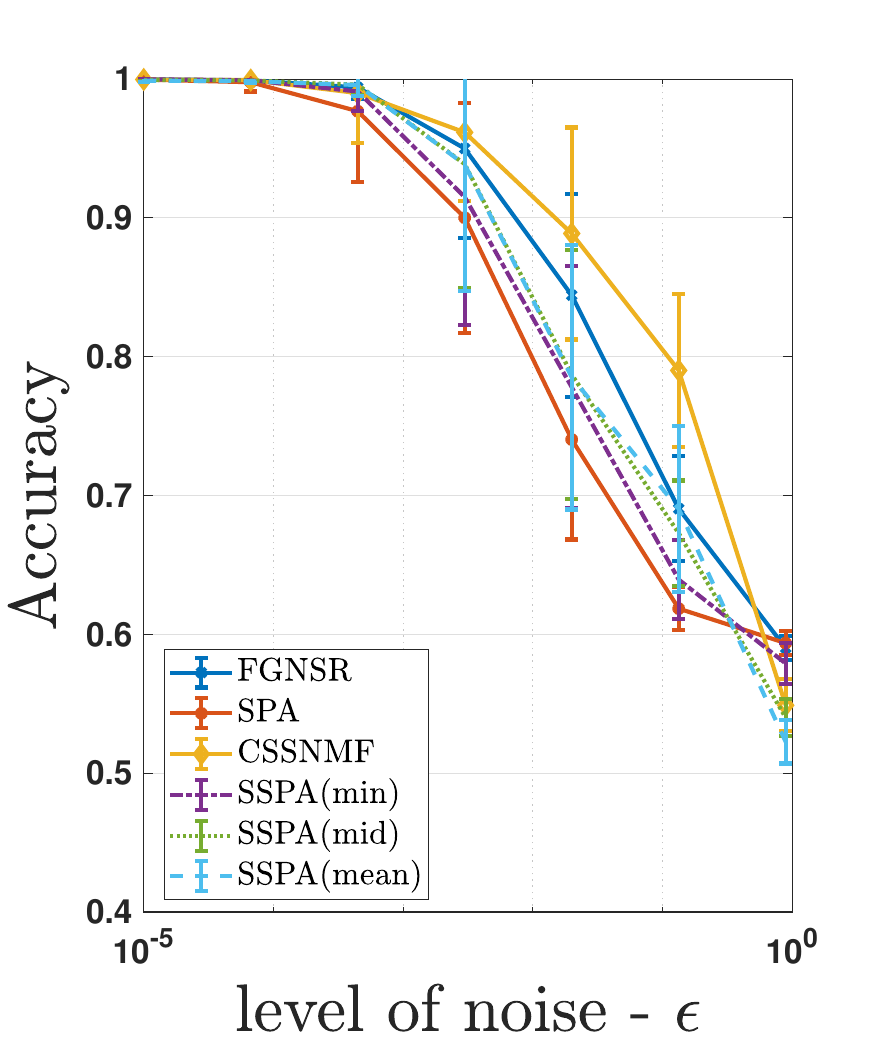}\hfill
  \includegraphics[width=0.33\textwidth]{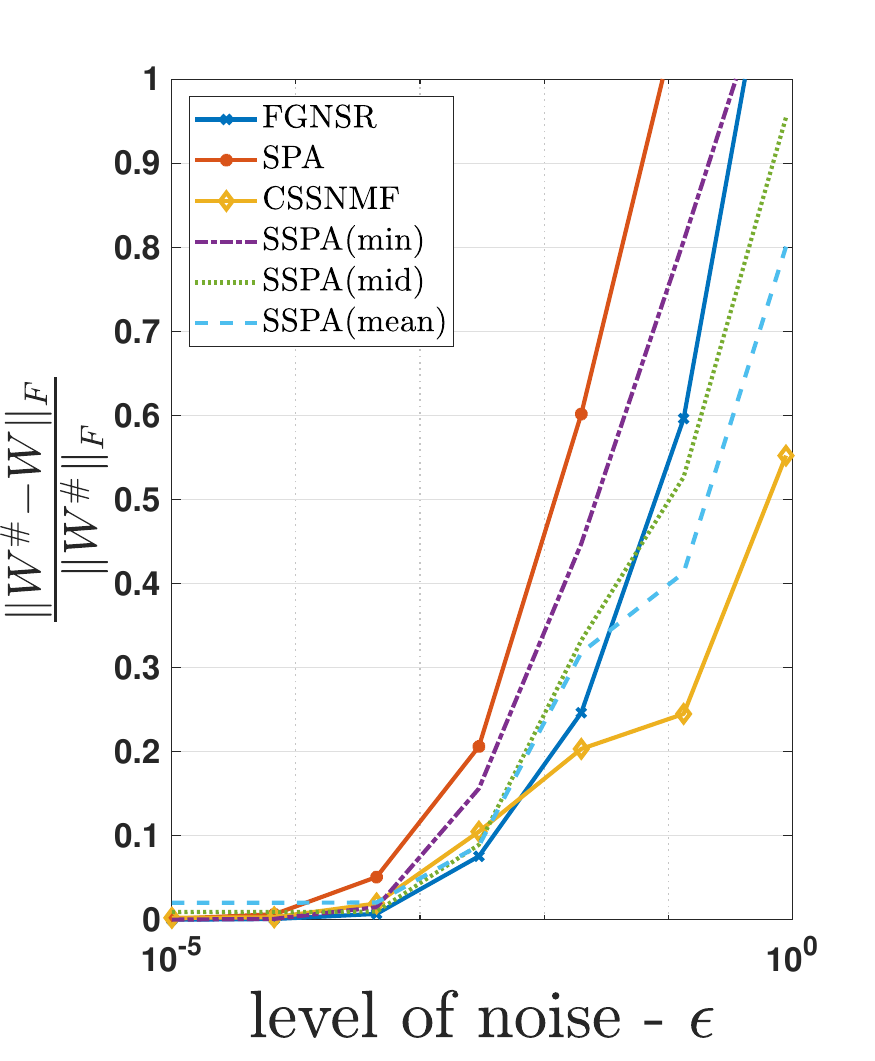}
  \caption{Average (over trials) for Dirichlet mixtures: average results across trials for FGNSR, SPA, CSSNMF, SSPA(min), SSPA(mid), SSPA(mean)}
  \label{fig:rand_avg}
\end{figure}

\subsubsection{Middle points with adversarial noise}\label{sec:advnoise}

We consider an adversarial setting designed to emphasize anchor point selection.
The data matrix $M\in \mathbb{R}^{30\times 95}$ is generated by
\[
M \;=\; \max(0,\, M_0 + N), \qquad M_0 \;=\; W[H_0,\,H_1],
\]
with the setting:
\begin{itemize}

\item The block \(H_0 \in \mathbb{R}^{10\times 50}\) is one-hot (i.e., each column is a canonical vector). It is constructed such that the columns of $ M_0 $ are partitioned into $ r=10 $ classes, with each class containing exactly $ p_t = n_0 / r = 5 $ columns for $ t = 1, \dots, r $. This implies that the columns of $M_0$ indexed by \(H_0\) are pure replicates of the \(r\) vertices \(W(:,t)\) where $t=1,\cdots, r$.
The block \(H_1 \in \mathbb{R}^{10\times 45}\) contains pairwise midpoints: for each unordered pair \(\{p,q\}\subset\{1,\dots,r\}\),
we include the column \(\tfrac{1}{2}e_p + \tfrac{1}{2}e_q\). All the \(\binom{r}{2}=45\)   pairs are used to generate the columns of $H_1$.

\item We add adversarial noise only to the midpoint block:
\[
N(:,1:n_0) = 0, \qquad
N(:,n_0+1:n) = M_0(:,n_0+1:n) - \bar w\, e^\top,\quad
\bar w \;=\; \tfrac{1}{n_0}\,W H_0 e, 
\]
where $n_0=50, n=95$. 
Then we scale the matrix $N$ such that \(\|N\|_F = \epsilon \|M_0\|_F\) and project it onto the nonnegative orthant.
We normalize the columns of \(M\) to have unit $\ell_1$ norm. 


\end{itemize}

We consider 4 different noise levels $\epsilon$ logarithmically spaced in $[10^{-3},
 10^{-0.05}]$ (in MATLAB, logspace(-3, -0.05, 4)). For each noise level, we generate 10 such matrices and report the average quality measures on Figure \ref{fig:mid-avg-all}. We make the following observations:

\begin{enumerate}[label=(\roman*)]
\item \textbf{Accuracy:} Our CSSNMF maintains perfect accuracy (\texttt{100\%}) across all noise levels. The curves for FGNSR, SPA, and  SSPA(min)  completely overlap, achieving \texttt{100\%} accuracy for $\epsilon < 0.1$. Beyond $\epsilon = 0.1$, their accuracy declines. The SSPA(mid)  and  SSPA(mean)  variants exhibit nearly identical performance, consistently below CSSNMF.

\item \textbf{Reconstruction error and relative $W$-error:} CSSNMF  achieves the \textit{lowest} values in both metrics across all noise levels, demonstrating superior robustness. The second best is FGNSR, then followed by SPA and SSPA(min), which perform comparably. In contrast, SSPA(mid) and SSPA(mean) yield higher errors, making them less effective under this adversarial noise setting.
\end{enumerate}

%

 \begin{figure}[t]
  \centering
  \includegraphics[width=0.33\textwidth]{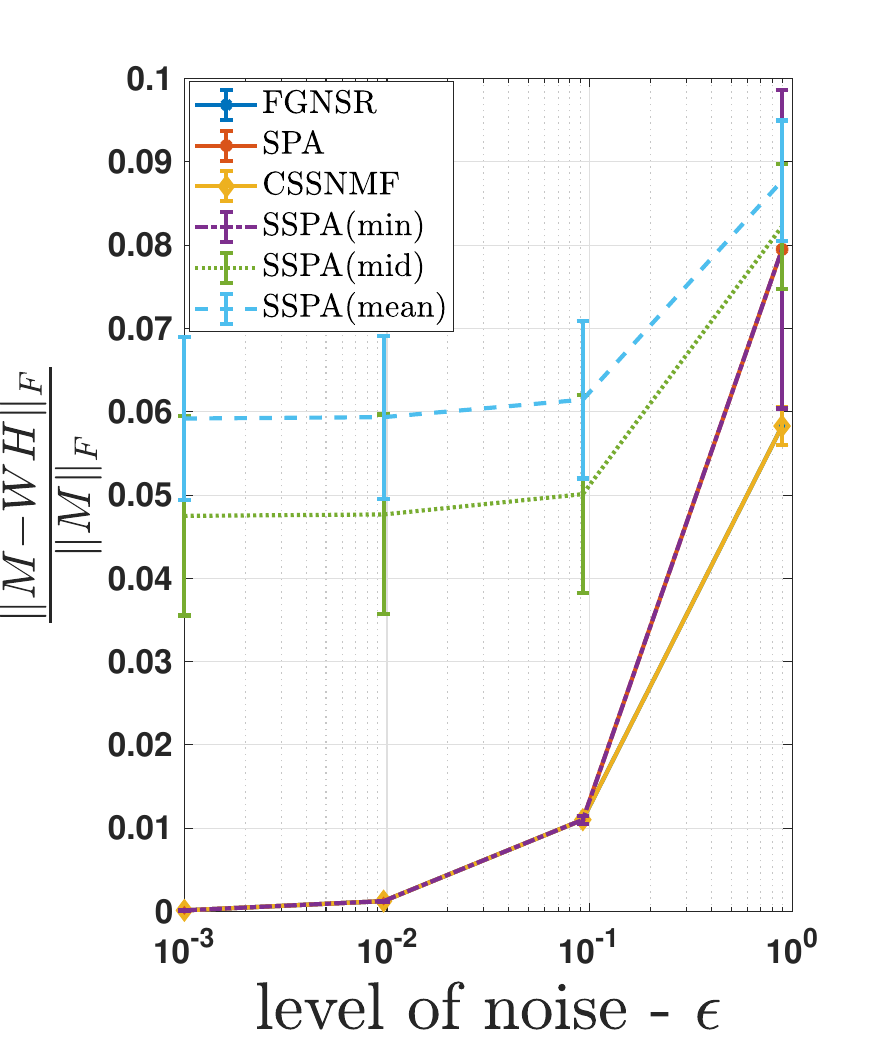}\hfill
 \includegraphics[width=0.33\linewidth]{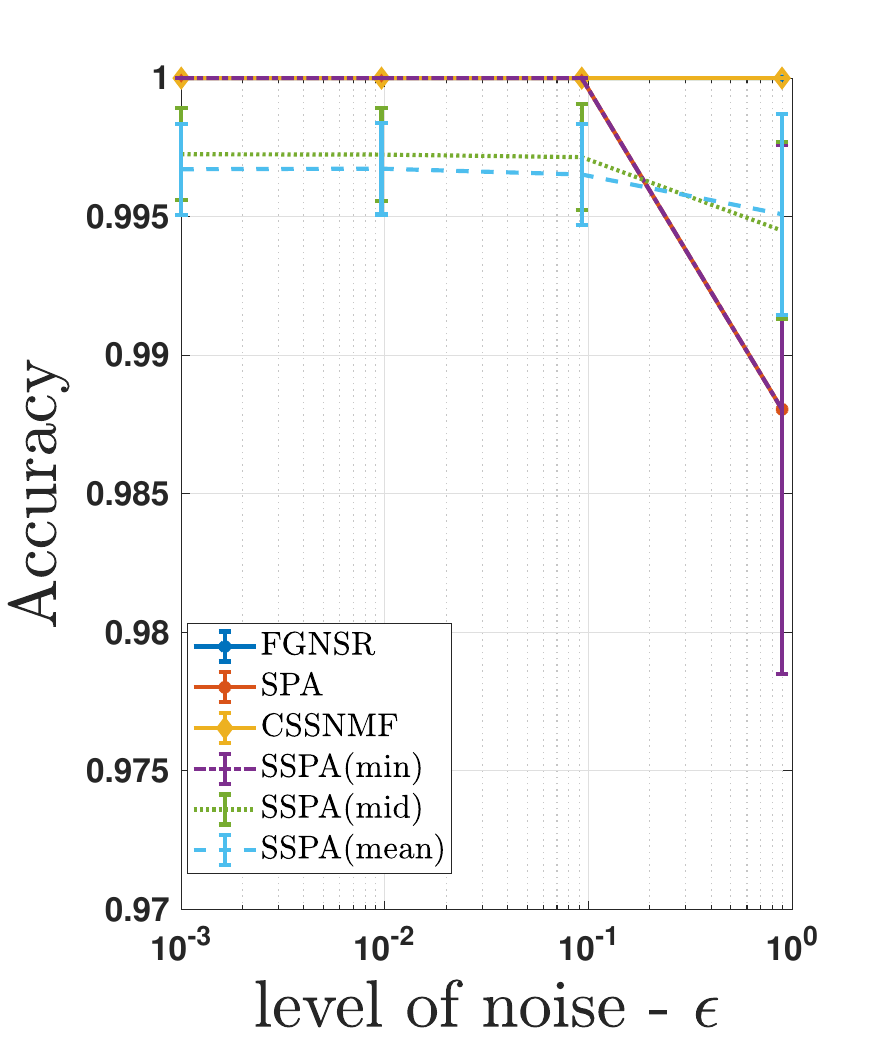}\hfill
\includegraphics[width=0.33\linewidth]{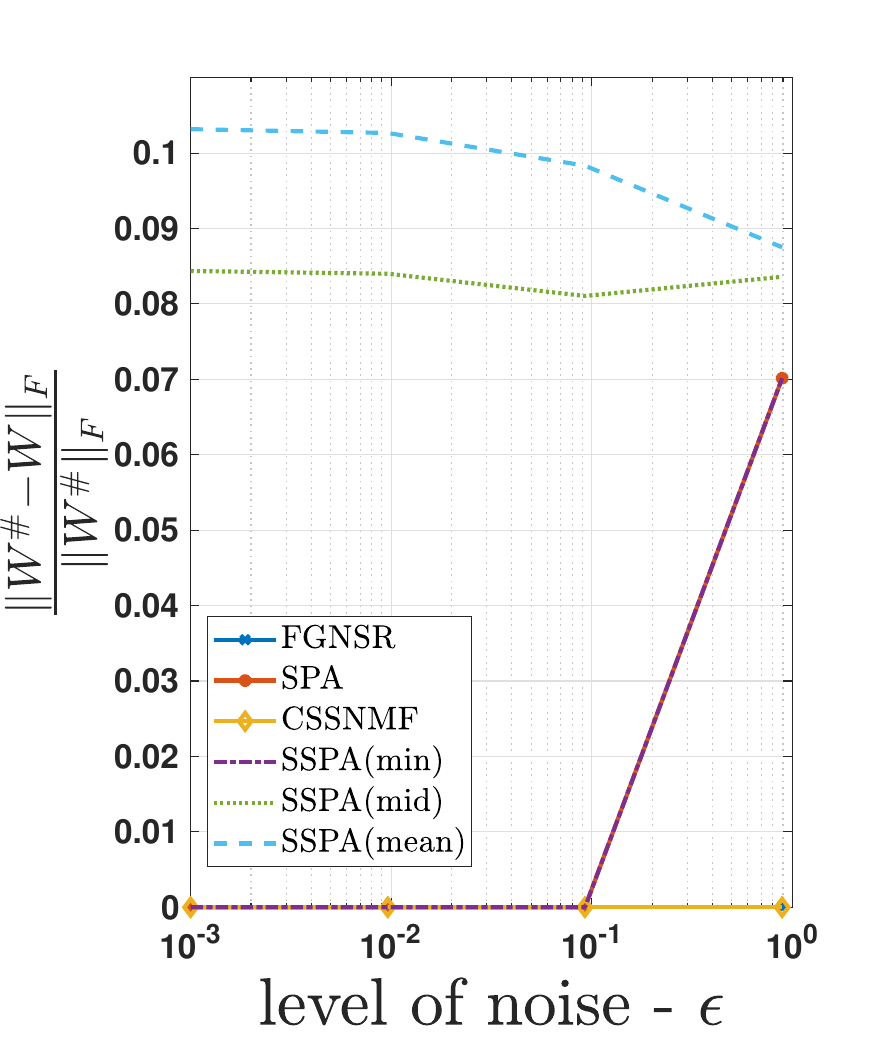}
  \caption{Middle points (adversarial noise): average results across trials for FGNSR, SPA, CSSNMF, SSPA(min), SSPA(mid), SSPA(mean).}
  \label{fig:mid-avg-all}
\end{figure}

\subsubsection{Robustness to outliers (median vs.\ average aggregation)}
\label{sec:outliers}

To test robustness to outliers, the data matrix $M\in \mathbb{R}^{30\times 65}$ is generated by
\[
M \;=\; [ M_0,\, B], \qquad M_0 \;=\; WH_0,  
\]
with the setting below:
\begin{itemize}

\item The block \(H_0 \in \mathbb{R}^{5\times 50}\) is one-hot (i.e., each column is a canonical vector). The $H_0$ is generated to guarantee the columns of $M_0$ belong to $r=5$ classes with each class size is $p_t=n_0/r=10$ for $t=1,\dots,r$. 

\item We append $\ell$ \emph{outlier} columns with the entries drawn i.i.d.\ in $[0,1]$ to form $B\!\in\!\mathbb{R}^{30\times 15}$
and construct $M$, then we apply columnwise $\ell_1$ normalization on $M$ such that the sum of each column of $M$ is 1.

\end{itemize}
We run our CSSNMF algorithm and compare two
aggregation rules to produce $\widehat W$ in the clustering/aggregation step: (i) \textbf{average} vs.\ (ii) \textbf{median}. We generate 10 such matrices and sweep the number of outliers $\ell$ in $\{1,\dots,15\}$  to report the average and best quality measures over trials on Figure \ref{fig:outliers-avg-best}. 
We have the following observations:

\begin{enumerate}[label=(\roman*)]
\item  \emph{Best over trials:} the \textbf{median} aggregation achieves relative W-error $d_W=0$ for
\emph{all} $\ell$ in the sweep, while the \textbf{average} aggregation exhibits a monotone increase of $d_W$ as $\ell$ grows.

\item \emph{Average over trials:} as predicted by our outlier-robustness corollary,
the \textbf{median} stays at $d_W=0$ as long as $\ell<\min_t p_t=10$, then becomes nonzero for $\ell\ge 10$ but remains consistently below the \textbf{average} aggregation, whose error increases with~$\ell$.
\end{enumerate}

\begin{figure}[t]
  \centering
  \begin{minipage}[t]{0.48\textwidth}
    \centering
    \includegraphics[width=\linewidth]{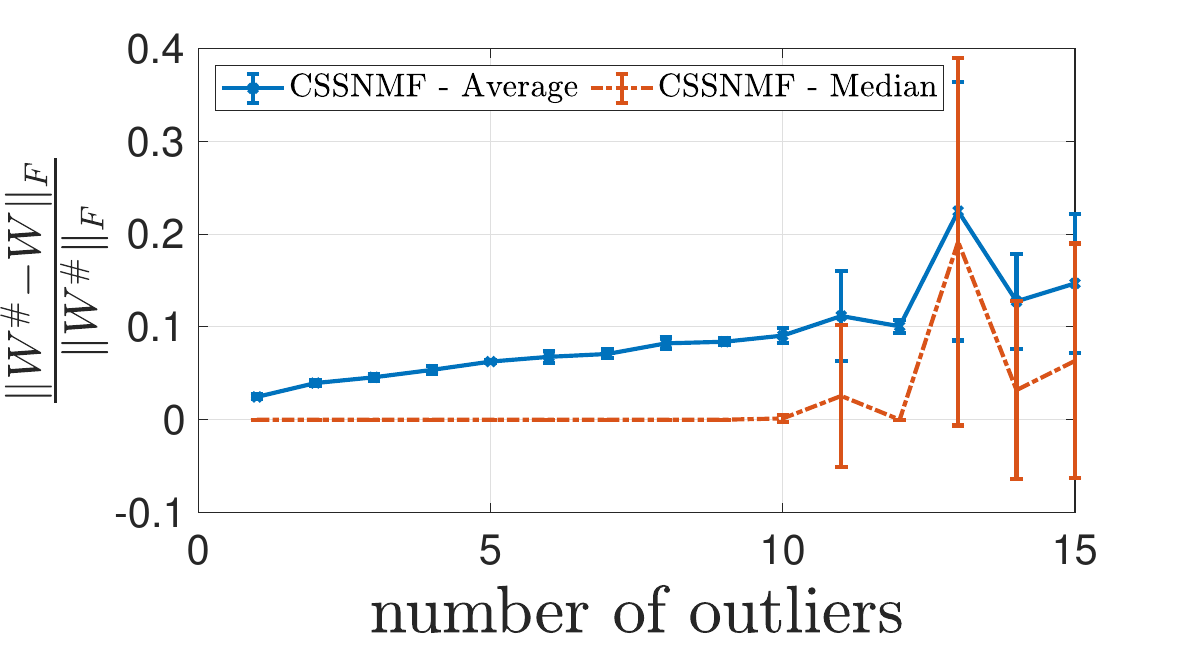}\\[-2pt]
    \small \emph{Average across trials:} relative $W$-error $d_W$ vs.\ number of outliers~$\ell$.
  \end{minipage}\hfill
  \begin{minipage}[t]{0.48\textwidth}
    \centering
    \includegraphics[width=\linewidth]{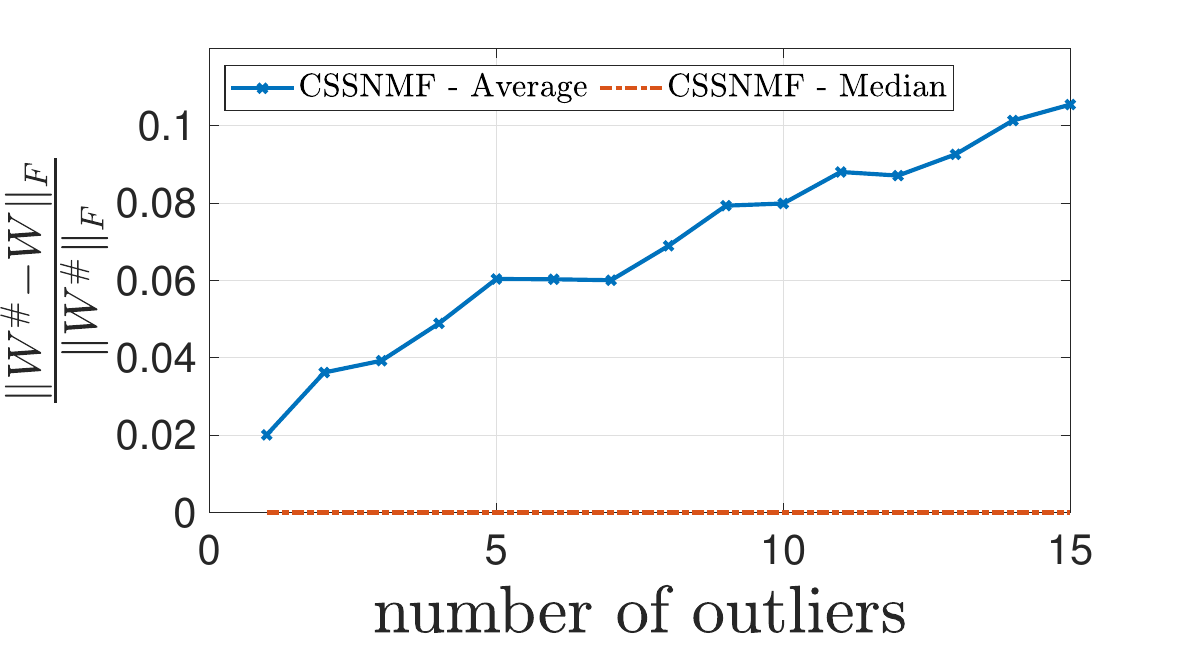}\\[-2pt]
    \small \emph{Best over trials:} relative $W$-error $d_W$ vs.\ number of outliers~$\ell$.
  \end{minipage}
  \caption{Outliers experiment: CSSNMF with \textbf{average} vs.\ \textbf{median} aggregation.}
  \label{fig:outliers-avg-best}
\end{figure}

\subsection{Blind Hyperspectral Unmixing}\label{sec:hyper}

Hyperspectral imaging (HSI) captures a detailed reflectance spectrum for every pixel in a scene across hundreds of contiguous wavelength bands. This results in a three-dimensional data cube that can be reshaped into a nonnegative matrix \( M \in \mathbb{R}^{m \times n} \), where \( m \) is the number of spectral bands and \( n \) is the total number of pixels. Each column \( M(:,j) \) represents the spectral signature (reflectance profile) of pixel \( j \). Hyperspectral unmixing aims to decompose this observed mixture \( M \) into its constituent pure materials and their spatial distributions. Under the widely adopted linear mixing model, each pixel spectrum is assumed to be a convex combination of a small number \( r \) of pure material spectra (called endmembers), weighted by their fractional abundances in that pixel. Mathematically, \( M \approx WH \), where 
 \( W \in \mathbb{R}^{m \times r} \) has columns that are the endmember signatures; and \( H \in \mathbb{R}^{r \times n} \) has rows that are the corresponding abundance maps.


We evaluate on three publicly available HSI benchmarks from the Lesun data repository\footnote{\url{https://lesun.weebly.com/hyperspectral-data-set.html}}:
\begin{itemize}
\item Jasper Ridge  (\( r=4 \) endmembers, \( m=198 \) spectral bands);
\item Urban  (\( r=6 \)  endmembers, \( m=162 \) spectral bands);
\item Samson  (\( r=3 \)  endmembers, \( m=156 \) spectral bands,  on a \( 95 \times 95 \) spatial crop).
\end{itemize}
Ground-truth endmembers (and abundance maps where provided) are provided by the repository. Following standard practice, we normalize each dataset so that \( \max_{i,j} M_{ij} = 1 \) before factorization. No additional pre- or post-processing is applied (no band selection, denoising, dimensionality reduction, or outlier removal), except for the Urban dataset. To ensure feasible computation time on Urban, we estimate \( W \) using a regular \( 1/9 \) spatial subsample (selecting one pixel every three rows and columns on a grid), then compute the full-resolution \( H \) via nonnegative least squares (NNLS) with \( W \) is fixed. All reconstruction errors and performance metrics are reported on the complete high-resolution images, intentionally placing algorithms in a realistic and challenging evaluation setting.

%
%


 
In HSI experiments, since FGNSR is slow and the results \revise{are}  not competitive compared to SSPA, for simplicity, we do not present its results here. For SSPA, we tune the hyperparameter nplp (number of proximal latent points) via a coarse grid search, choosing the value that minimizes reconstruction error while ensuring the recovered endmembers and abundance maps remain visually plausible. 


We report three complementary measures including (i) the relative reconstruction error \eqref{error}  (in \%), (ii) the endmember error \eqref{W_error} (in \%), and
(iii) the structural similarity index (SSIM; in $[0,1]$) between each estimated abundance map and its ground truth, averaged over $r$ maps (when ground-truth maps are available).
Before computing endmember error  $d_W$, we match columns by permutation and normalize endmember columns to unit $\ell_1$ norm for scale invariance; SSIM is computed on per-map images rescaled to $[0,1]$. The results are summarized in Table~\ref{tab:hsi}. We observe that, 

\begin{enumerate}[label=(\roman*)]

\item In terms of the computational time, wall-clock runtimes for CSSNMF on the full hyperspectral cubes range from approximately 10 to 30 minutes. In contrast, SPA and SSPA complete in a few seconds each.

\item Across the three datasets, CSSNMF
achieves the lowest reconstruction error and the most accurate endmembers ($d_W$),
with the highest SSIM on abundance maps. SPA is competitive on endmembers when its anchors are cleanly identifiable such as in Samson, but degrades significantly on Jasper and on Urban; SSPA improves over SPA with tuned \texttt{nplp} but remains behind CSSNMF on average.  These gains come at higher compute for CSSNMF, as noted above.

\end{enumerate}

\begin{table}[t]
  \centering
  \setlength{\tabcolsep}{6pt}
  \begin{tabular}{|l|l|c|c|c|}
    \hline
    Dataset & Method & Rel.\ error (\%) $\downarrow$ & $d_W$ (\%) $\downarrow$ & SSIM (0-1) $\uparrow$ \\
    \hline
    \multirow{3}{*}{Jasper Ridge ($r{=}4$)}
      & CSSNMF  & \textbf{\,\;4.93\,} & \textbf{\,\;6.23\,} & \textbf{\,\;0.94\,} \\
      & SPA             & 8.69 & 48.47 & 0.53 \\
      & SSPA (tuned \texttt{nplp}) & 6.58 & 26.74 & 0.59 \\
    \hline
    \multirow{3}{*}{Urban ($r{=}6$)}
      & CSSNMF  & \textbf{6.64} & \textbf{12.7} & \textbf{0.71} \\
      & SPA             & 10.3 & 66.8 & 0.35 \\
      & SSPA (tuned \texttt{nplp}) & 8.01 & 43.9 & 0.40 \\
    \hline
    \multirow{3}{*}{Samson ($r{=}3$)}
      & CSSNMF  & \textbf{4.12} & 2.60 & \textbf{0.99} \\
      & SPA             & 6.07 & 11.4 & 0.94 \\
      & SSPA (tuned \texttt{nplp}) & 4.19 & \textbf{2.59} & 0.98 \\
    \hline
  \end{tabular}
  \caption{Blind HSI unmixing on three benchmarks.  Best values are highlighted in \textbf{bold}.}

  \label{tab:hsi}
\end{table}


Figure~\ref{fig:jasper-signatures} (resp. Figure~\ref{fig:jasper-abundances}) shows the ground-truth endmember signatures (resp.\ abundance maps) and their estimates by the three algorithms. 
 CSSNMF recovers sharper abundance maps and endmember spectra that adhere more closely to the ground truth. Due to the page limit, the visual results for Urban and Samson are available in Appendix~\ref{app:hsi}.

\begin{figure}[ht!]
  \centering
  \includegraphics[width=0.9\textwidth]{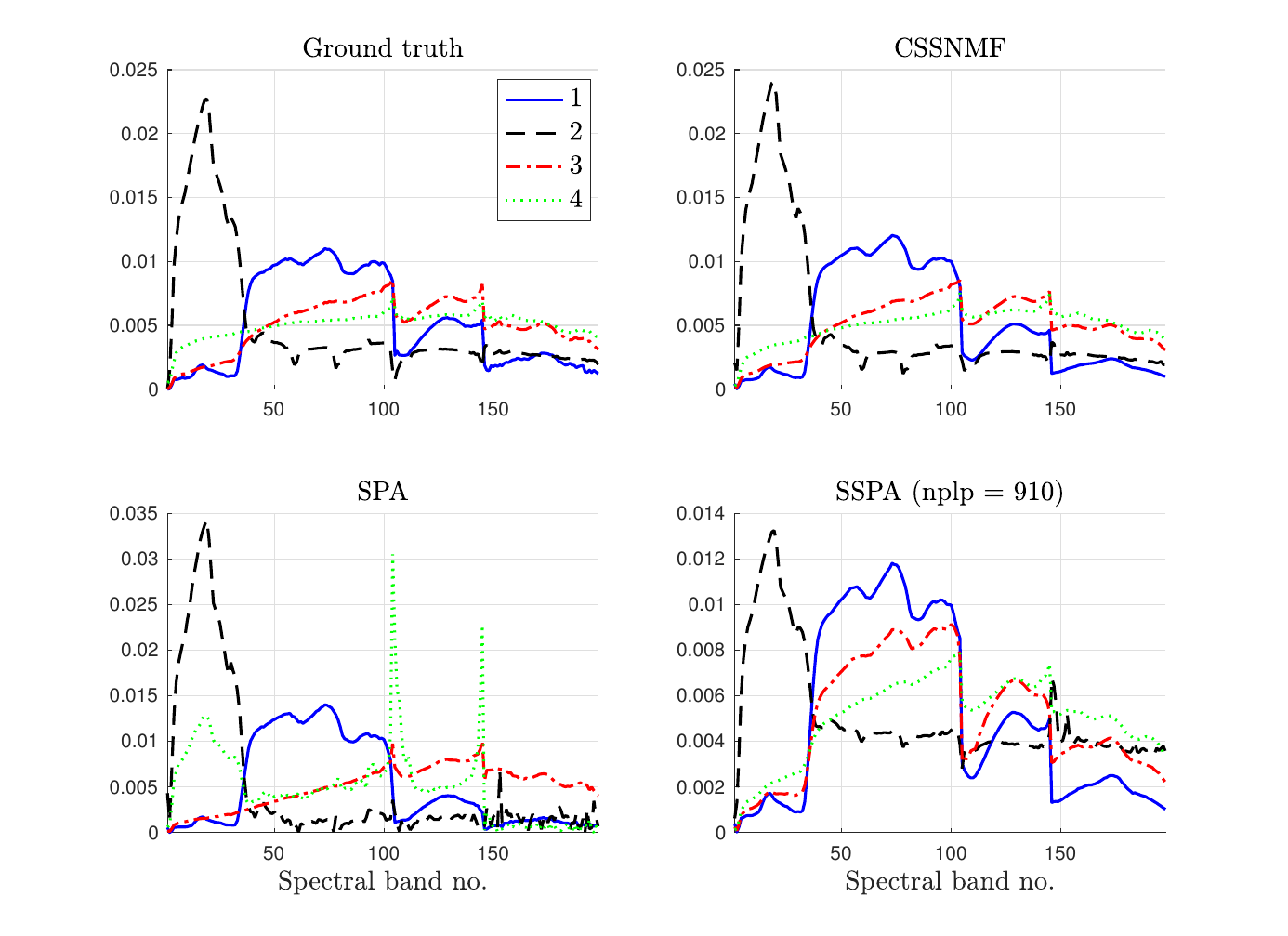}
  \caption{Jasper Ridge endmember spectral signatures: : ground truth (top left), CSSNMF (top right),  SPA (bottom left) and  SSPA with tuned \texttt{nplp} 
  (bottom right).}
  \label{fig:jasper-signatures}
\end{figure}

\begin{figure}[p] 
  \centering
  \includegraphics[
    height=0.23\textheight,
    keepaspectratio,
    width=\textwidth,
    trim=2.0cm 4.0cm 2.0cm 2.0cm, clip
  ]{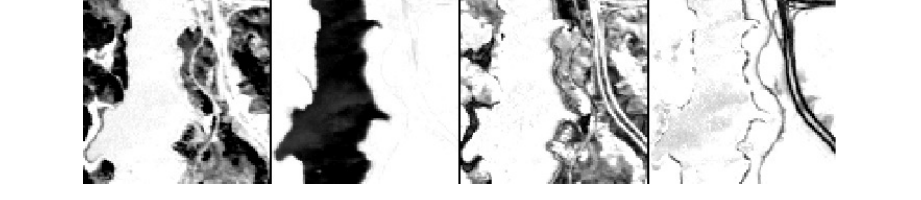}\vspace{0.3em}

  \includegraphics[
    height=0.23\textheight,
    keepaspectratio,
    width=\textwidth,
    trim=2.0cm 4.0cm 2.0cm 2.0cm, clip
  ]{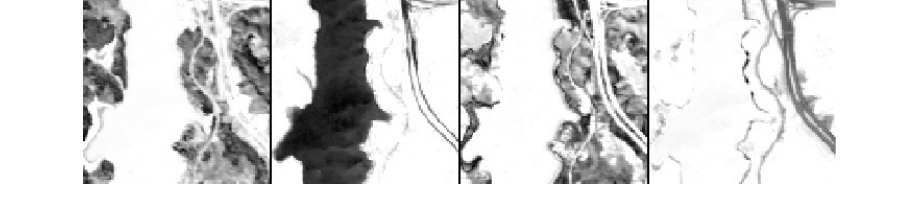}\vspace{0.3em}

  \includegraphics[
    height=0.23\textheight,
    keepaspectratio,
    width=\textwidth,
    trim=2.0cm 4.0cm 2.0cm 2.0cm, clip
  ]{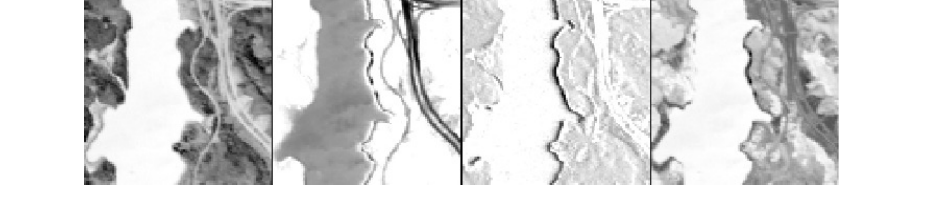}\vspace{0.3em}

  \includegraphics[
    height=0.23\textheight,
    keepaspectratio,
    width=\textwidth,
    trim=2.0cm 4.0cm 2.0cm 2.0cm, clip
  ]{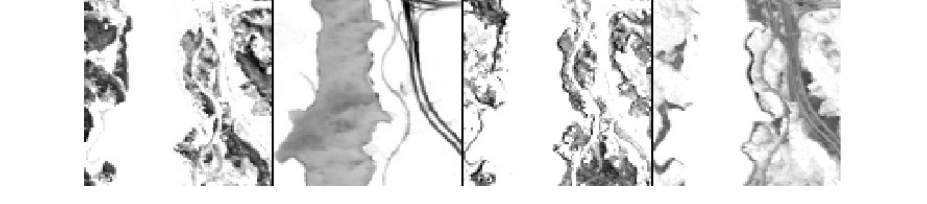}

  \caption{Abundance maps for Jasper Ridge. From top to bottom: ground truth, CSSNMF (Alg.~1 and 2), SPA, SSPA (tuned \texttt{nplp}).}
  \label{fig:jasper-abundances}
\end{figure}

\section{Conclusion}

In this paper, we introduced a convex SSNMF (CSSNMF) model for smooth separable NMF (SSNMF), leveraging the presence of multiple points near each basis vector. 
We also showed connections between orthogonal, SNMF, and SSNMF, providing a new insight on these models.  
We then provided theoretical guarantees for CSSNMF, both noiseless and noisy regimes. 
Our experiments on synthetic data and real hyperspectral images showed that CSSNMF is competitive with, and often ahead of, strong baselines (namely, SPA, SSPA, FGNSR).  

The main limitation of CSSNMF is its computational cost as it requires to solve a large convex optimization problem in $O(n^2)$ variables, where $n$ is the number of data points. On full hyperspectral images (without subsampling),  our current MATLAB pipeline typically takes on the order of 10--30 minutes per dataset on a recent laptop (MacBook Pro, Apple M4, 24\,GB RAM). 
We expect substantial speedups from tailored algorithms, such as a column-generation approach as done in~\cite{mizutani2025endmember} or first-order solvers, with warm starts and screening, block/stochastic updates, GPU/multi-threaded implementations, and sketching/coresets that preserve separability \revise{as done very recently in~\cite{mizutani2025hyperspectral}}. 
This is a topic for further research. 


\revise{
\section*{Acknowledgments}  We are grateful to the anonymous reviewers who
carefully read the manuscript, whose feedback helped improve the paper.  
}


\bibliographystyle{abbrv}
\bibliography{nmfref}

\normalsize 

\newpage 

\appendix

\section{Visual results for Urban and Samson hyperspectral images} \label{app:hsi} 

This supplementary material provides the spectral signatures and abundance maps extracted by the algorithms compared in the paper, as well as the ground truth,  
for the Urban and Samson hyperspectral images. 

\begin{figure}[ht!]
  \centering
  \includegraphics[width=\textwidth]{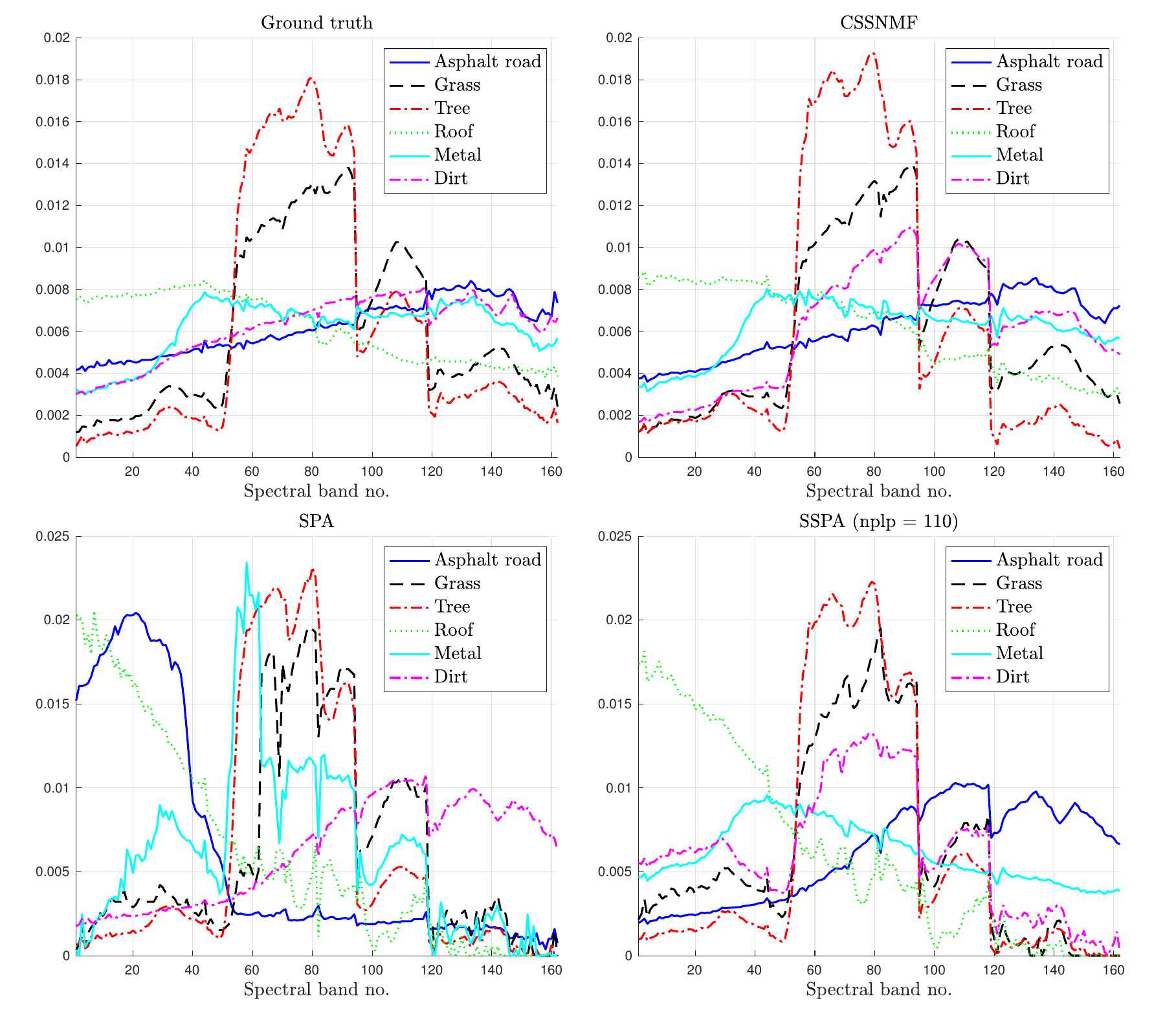}
  \caption{Urban endmember spectral signatures: ground truth (top left), CSSNMF (top right),  SPA (bottom left) and  SSPA tuned \texttt{nplp} 
  (bottom right).}
  \label{fig:urban-signatures}
\end{figure}

\begin{figure}[p] 
  \centering
  \includegraphics[
    height=0.23\textheight,
    keepaspectratio,
    width=\textwidth,
    trim=1.0cm 1.0cm 1.0cm 1.0cm, clip
  ]{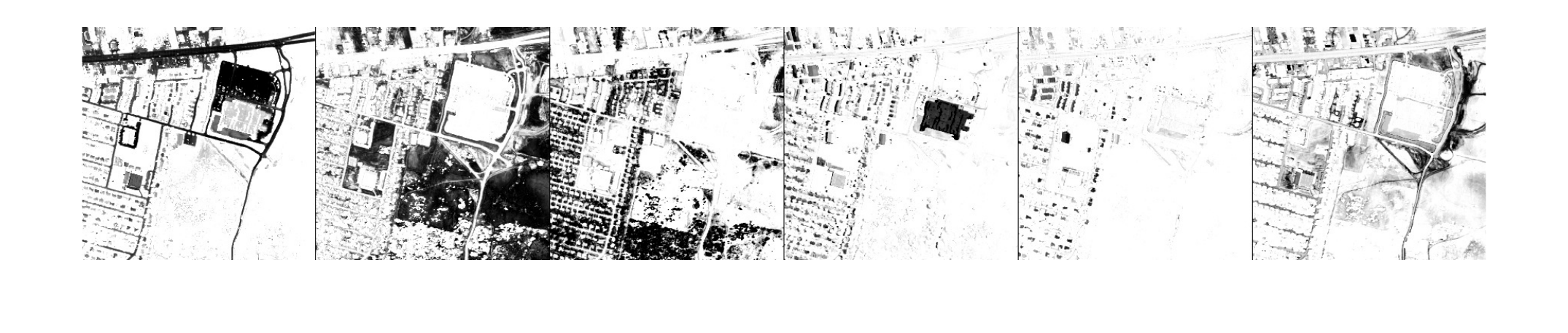}\vspace{0.3em}

  \includegraphics[
    height=0.23\textheight,
    keepaspectratio,
    width=\textwidth,
    trim=1.0cm 1.0cm 1.0cm 1.0cm, clip
  ]{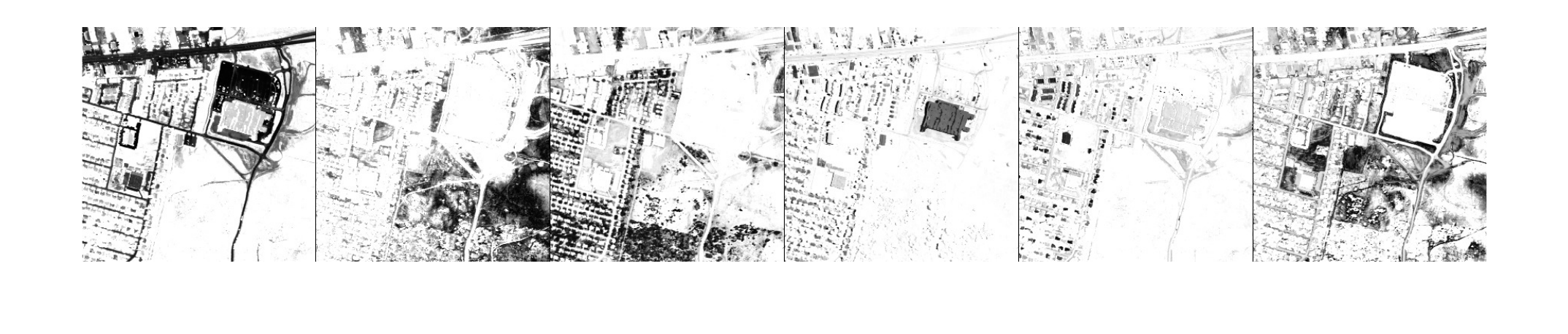}\vspace{0.3em}

  \includegraphics[
    height=0.23\textheight,
    keepaspectratio,
    width=\textwidth,
    trim=1.0cm 1.0cm 1.0cm 1.0cm, clip
  ]{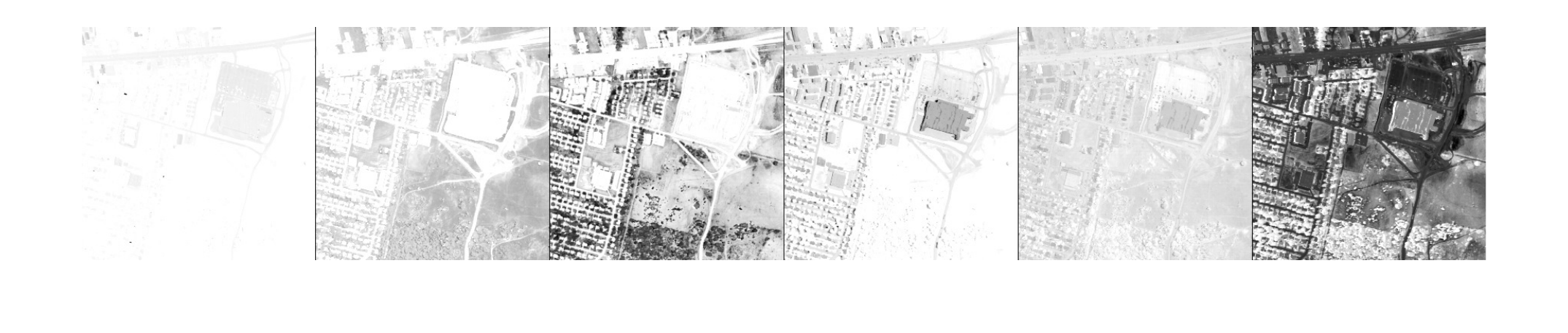}\vspace{0.3em}

  \includegraphics[
    height=0.23\textheight,
    keepaspectratio,
    width=\textwidth,
    trim=1.0cm 1.0cm 1.0cm 1.0cm, clip
  ]{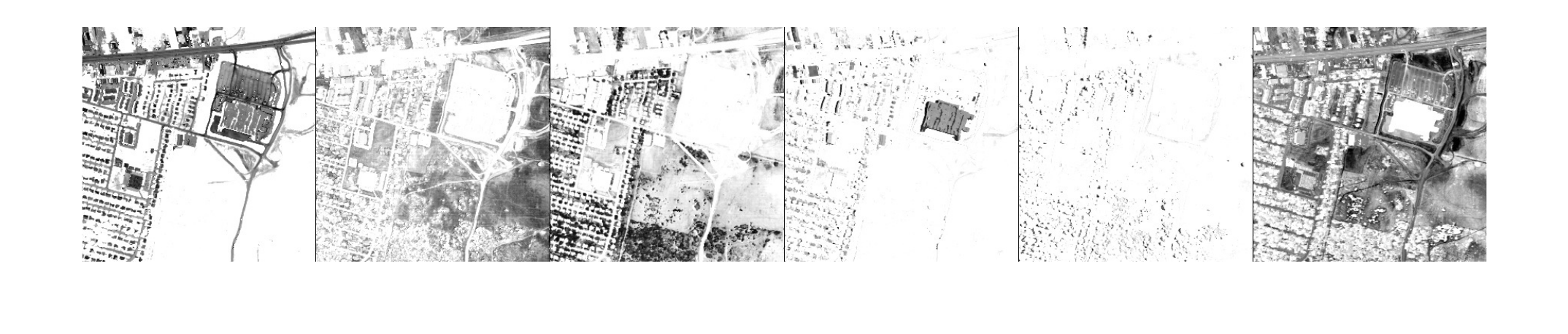}

  \caption{Abundance maps for Urban. From top to bottom: ground truth, CSSNMF, SPA, SSPA (tuned \texttt{nplp}).}
  \label{fig:urban-abundances}
\end{figure}

\begin{figure}[ht!]
  \centering
  \includegraphics[width=\textwidth]{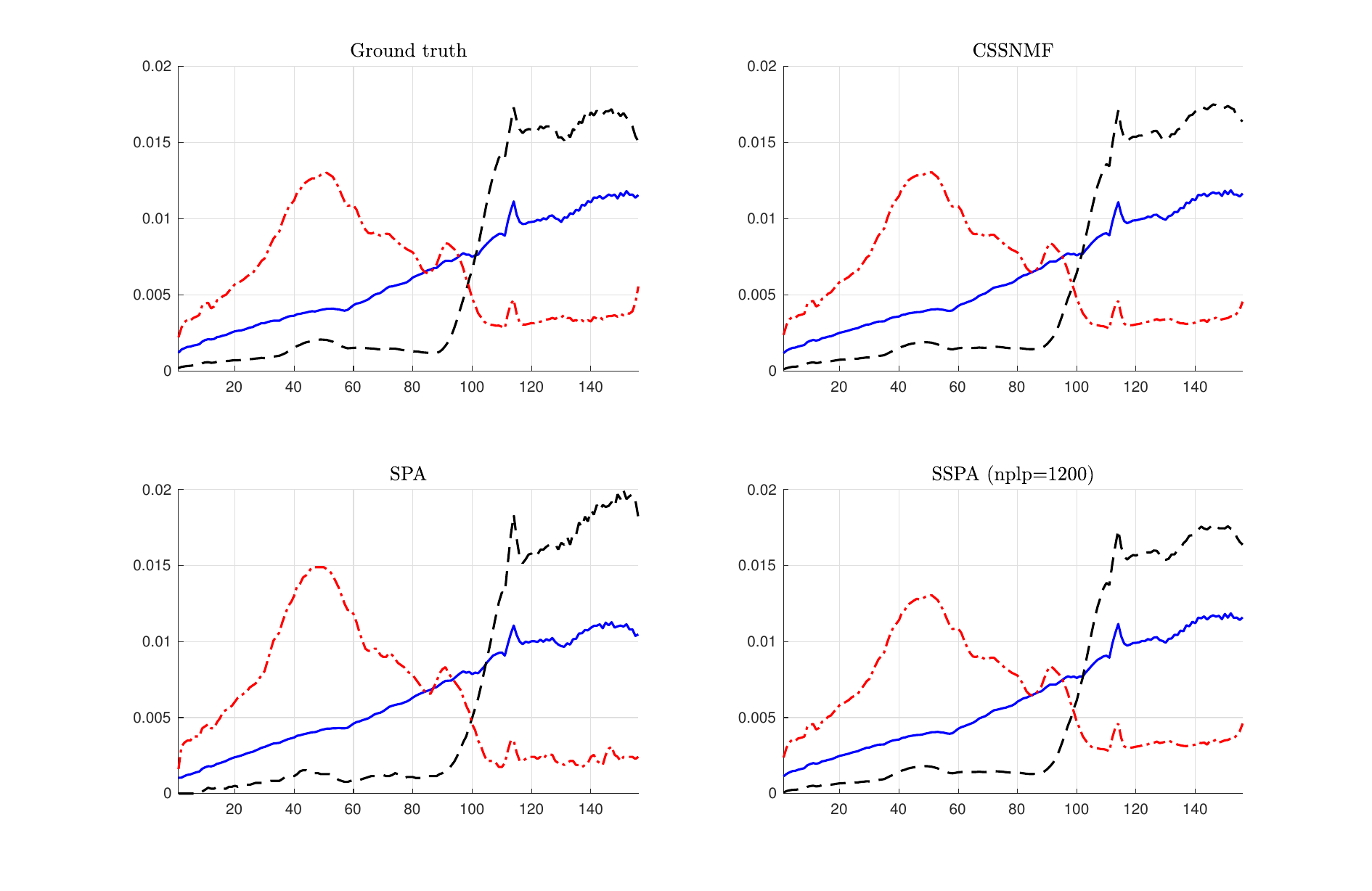}
  \caption{Samson endmember spectral signatures: ground truth (top left), CSSNMF (top right),  SPA (bottom left) and  SSPA tuned \texttt{nplp} 
  (bottom right).}
  \label{fig:samson-signatures}
\end{figure}

\begin{figure}[p] 
  \centering
  \includegraphics[
    height=0.23\textheight,
    keepaspectratio,
    width=\textwidth,
    trim=2.0cm 4.0cm 2.0cm 2.0cm, clip
  ]{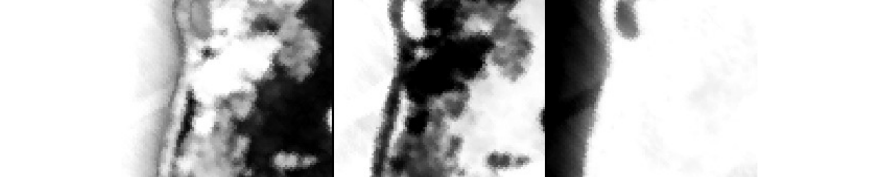}\vspace{0.3em}

  \includegraphics[
    height=0.23\textheight,
    keepaspectratio,
    width=\textwidth,
    trim=2.0cm 4.0cm 2.0cm 2.0cm, clip
  ]{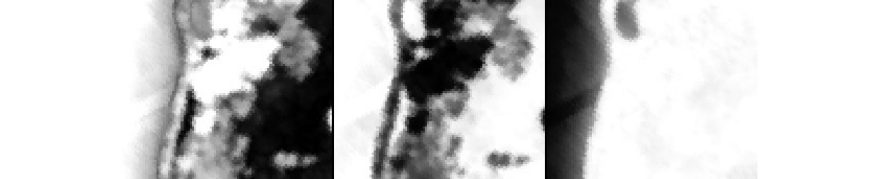}\vspace{0.3em}

  \includegraphics[
    height=0.23\textheight,
    keepaspectratio,
    width=\textwidth,
    trim=2.0cm 4.0cm 2.0cm 2.0cm, clip
  ]{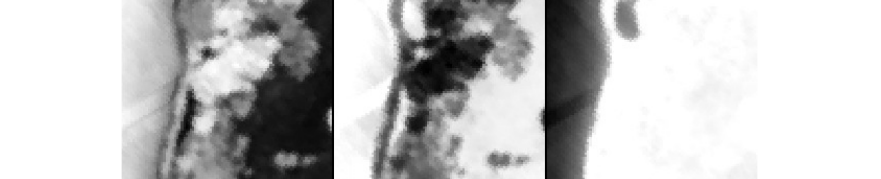}\vspace{0.3em}

  \includegraphics[
    height=0.23\textheight,
    keepaspectratio,
    width=\textwidth,
    trim=2.0cm 4.0cm 2.0cm 2.0cm, clip
  ]{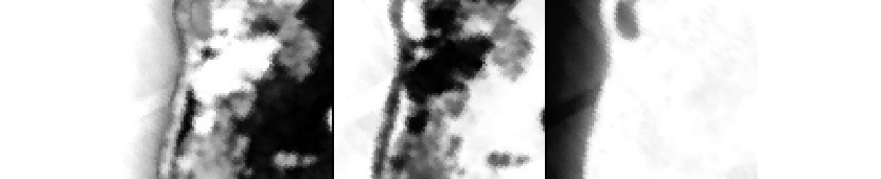}

  \caption{Abundance maps for Samson. From top to bottom: ground truth, CSSNMF, SPA, SSPA (tuned \texttt{nplp}).}
  \label{fig:samson-abundances}
\end{figure}

\end{document}